\documentclass[cis]{ipart_v1}

\firstpage{1}
\Vol{X}
\Issue{X}
\Year{2020}

\usepackage{enumerate}
\usepackage{physics}
\usepackage{dsfont}
\usepackage{tikz}
\usepackage{color}
\usetikzlibrary{matrix,decorations.pathreplacing}
\usepackage{mathrsfs}
\usepackage{csquotes}
\usepackage{graphicx}

\theoremstyle{plain}
\newtheorem{theorem}{Theorem}[section]
\newtheorem{proposition}{Proposition}[section]
\newtheorem{lemma}{Lemma}[section]

\newtheorem{lemma_1}{Lemma}[section]
\theoremstyle{definition}
\newtheorem{Definition}{Definition}[section]

\theoremstyle{definition}
\newtheorem{Remark}{Remark}[section]
\newtheorem{Example}{Example}[section]
\newtheorem{corollary}{Corollary}[section]

\newtheorem{Remark_1}{Remark}[section]

\begin{document}

\title[Zero-Sum Differential Games on the Wasserstein Space]{Zero-Sum Differential Games on the Wasserstein Space$^\ast$}

\author[Jun Moon and Tamer Ba\c{s}ar]{Jun Moon$^\dag$, Tamer Ba\c{s}ar$\ddagger$
\blfootnote{$^\ast$Dedicated to Tyrone Duncan on the occasion of his 80th birthday.}
\blfootnote{
$^\dag$Research of Jun Moon was supported in part by the National Research Foundation of Korea (NRF) Grant funded by the Ministry of Science and ICT, Korea (NRF-2017R1E1A1A03070936, NRF-2017R1A5A1015311). 
%, and in part by Institute for Information \& communications Technology Promotion (IITP) grant funded by the Korea government, Korea (No. 2018-0-00958).
}
\blfootnote{
$\ddagger$Research of Tamer Ba\c{s}ar was supported in part by the Office of Naval Research (ONR) MURI Grant N00014-16-1-2710, and in part by the Air Force Office of Scientific Research (AFOSR) Grant FA9550-19-1-0353.}
}

\begin{abstract}
We consider two-player zero-sum differential games (ZSDGs), where the state process (dynamical system) depends on the random initial condition and the state process's distribution, and the objective functional includes the state process's distribution and the random target variable. Unlike ZSDGs studied in the existing literature, the ZSDG of this paper introduces a new technical challenge, since the corresponding (lower and upper) value functions are defined on $\mathcal{P}_2$ (the set of probability measures with finite second moments) or $\mathcal{L}_2$ (the set of random variables with finite second moments), both of which are infinite-dimensional spaces. We show that the (lower and upper) value functions on $\mathcal{P}_2$ and $\mathcal{L}_2$ are equivalent (law invariant) and continuous, satisfying dynamic programming principles. We use the notion of derivative of a function of probability measures in $\mathcal{P}_2$ and its lifted version in $\mathcal{L}_2$ to show that the (lower and upper) value functions are unique viscosity solutions to the associated (lower and upper) Hamilton-Jacobi-Isaacs equations that are (infinite-dimensional) first-order PDEs on $\mathcal{P}_2$ and $\mathcal{L}_2$, where the uniqueness is obtained via the comparison principle. Under the Isaacs condition, we show that the ZSDG has a value.
%Numerical examples are provided to illustrate the theoretical results of the paper.
\end{abstract}

\maketitle

\section{Introduction}
In this paper, we consider a class of nonlinear two-player zero-sum differential games (ZSDGs), where the state process (dynamical system) depends on the random initial condition and the state process's distribution (the law of the state process), and the objective functional includes the state process's distribution and the random target variable. The main objectives of this paper are to establish dynamic programming principles (DPPs) for lower and upper value functions of the ZSDG, show that the value function is the unique viscosity solution of the associated Hamilton-Jacobi-Isaacs (HJI) equation, and prove that under the Isaacs condition, the ZSDG has a value.

(Deterministic and stochastic) ZSDGs and their applications have been studied extensively in the literature; see \cite{Basar2, Bardi_book_1997, Bensoussan_book_2007, Buckdahn_Cardaliaguet_DGAA_2011, Basar_Zaccour_book} and the references therein. Specifically, Rufus Isaacs \cite{Isaacs_book} was the first who considered (deterministic) ZSDGs with applications to pursuit-evasion games. Elliott and Kalton \cite{Elliott_1972} introduced the concept of \emph{nonanticipative strategies} for the players, which was used in \cite{Evans_1984, Fleming_1989} to obtain DPPs for lower and upper value functions of the ZSDG, and show that the value functions are viscosity solutions to associated (lower and upper) HJI equations. The existence of the value of ZSDGs and the existence of saddle-point solutions were studied in \cite{Elliott_1972, Friedman_JDE_1970, Ghosh_JOTA_2004}.

Later, the results of \cite{Evans_1984, Fleming_1989} were extended to various other settings for ZSDGs. We mention here a few references that are relevant to our paper. The papers \cite{Buckdahn_SICON_2008, Li_AMO_2015} considered the class of games where the state and the objective functional are described by coupled forward-backward stochastic differential equations (SDEs). They used the so-called \emph{backward semigroup} associated with the backward SDE to obtain DPPs and the viscosity solution property of the HJI equations. ZSDGs with unbounded controls were considered in \cite{Qui_ESIAM_2013}. The weak formulation of ZSDGs and the mean field framework of ZSDGs were studied in \cite{Pham_SICON_2014, Li_MIn_SICON_2016, Djehiche_AMO_2018, Averboukh_DGAA_2018}, where in \cite{Pham_SICON_2014} path-dependent HJI equations and their viscosity solutions were considered. Reference \cite{Averboukh_DGAA_2018} used the feedback approach to construct a suboptimal solution and prove the existence of the value function. (Deterministic and stochastic) linear-quadratic ZSDGs with Riccati equations were studied in \cite{Zhang_SICON_2005, Delfour_SICON_2007, YU_SICON_2015, Sun_SICON_2016, Moon_TAC_2019_Markov} and the references therein. Maximum principles for risk-sensitive ZSDGs were established in \cite{Moon_TAC_Risk_2019}, and for nonzero-sum DGs in \cite{5446378}.

There are numerous applications of ZSDGs. Pursuit-evasion games and their applications to characterization of reachable sets for dynamical systems were considered in \cite{Basar_Zaccour_book, Buckdahn_Cardaliaguet_DGAA_2011,  Mitchell_TAC_2005,Margellos_TAC_2011}. 
%Mean-field games and mean-field type control problems, which arise in optimal complex-decision making for large population interacting particle systems, have been considered in \cite{Lasry, Carmona_book_2018, Andersson_AMO_2010, 7047683, Cardaliaguet, Rene_SICON_2013, Yong_SICON_2013_MF, Bensoussan_Arxiv_2014, Rene_anals_2015, Cardaliaguet_MFE_2018, Djehiche_AMO_2018, Averboukh_DGAA_2018, Moon_DGAA_2019} and the references therein. 
Optimal resource allocation, distributed control problems and their applications can also be considered within the framework of ZSDGs \cite{Yong_Book_2015, Basar_Zaccour_book}. For particular applications of ZSDGs considered in this paper, see the discussion in Examples \ref{Example_1}-\ref{Example_3} of Section \ref{Section_2_2}.

We should note that for the (lower and upper) value functions of ZSDGs and the associated HJI equations studied in \cite{Bardi_book_1997, Evans_1984, Fleming_1989, Buckdahn_SICON_2008, Li_AMO_2015, Qui_ESIAM_2013, Pham_SICON_2014}, the state space is the standard finite-dimensional space.
%\footnote{It can be $\mathbb{R}^n$ or a subset of $\mathbb{R}^n$.} 
In our formulation, however, the random initial condition and the law of the state process affect the dynamical system, and the objective functional includes the state process's distribution as well as the random target variable. Hence, unlike \cite{Bardi_book_1997, Evans_1984, Fleming_1989, Buckdahn_SICON_2008, Li_AMO_2015, Qui_ESIAM_2013, Pham_SICON_2014}, in this paper the state arguments of the (lower and upper) value functions and the associated HJI equations belong to $\mathcal{P}_2$ (the set of probability measures with finite second moments) and $\mathcal{L}_2$ (the set of random variables with finite second moments) that are \emph{infinite-dimensional spaces}. This inherent infinite-dimensional feature introduces a new technical challenge, which has not arisen in \cite{Bardi_book_1997, Evans_1984, Fleming_1989, Buckdahn_SICON_2008, Li_AMO_2015, Qui_ESIAM_2013, Pham_SICON_2014}. This is the challenge we tackle in this paper.\footnote{The ZSDG of the paper is closely related to mean field type games studied in \cite{Li_MIn_SICON_2016, Djehiche_AMO_2018, Averboukh_DGAA_2018} (see Example \ref{Example_1} in Section \ref{Section_2_2}). However, \cite{Li_MIn_SICON_2016, Djehiche_AMO_2018, Averboukh_DGAA_2018} considered weak and open-loop formulations, where the DPPs, the HJI equations, and the viscosity solution property of the value functions naturally do not arise. The problem formulation, the approach used, and the main results of this paper are different from \cite{Li_MIn_SICON_2016, Djehiche_AMO_2018, Averboukh_DGAA_2018}.}

The first main objective of the paper is to show that the (lower and upper) value functions defined on $[t,T] \times \mathcal{P}_2$ and $[t,T] \times \mathcal{L}_2$, where $[t,T]$ is a fixed time horizon, are equivalent to each other. This leads to the law invariant property between the value functions on $[t,T] \times \mathcal{P}_2$ and $[t,T] \times \mathcal{L}_2$, which were not considered in the finite-dimensional case in \cite{Bardi_book_1997, Evans_1984, Fleming_1989, Buckdahn_SICON_2008, Li_AMO_2015, Qui_ESIAM_2013, Pham_SICON_2014}. The (lower and upper) value functions on $[t,T] \times \mathcal{L}_2$ are called the lifted value functions. We also show that the (lower and upper) value functions on $[t,T] \times \mathcal{P}_2$ and their lifted version on $[t,T] \times \mathcal{L}_2$ are continuous. The proof of the continuity utilizes properties of the $2$-Wasserstein metric and the flow (semigroup) property of the state distribution, which are not needed in the finite-dimensional cases studied in \cite{Bardi_book_1997, Evans_1984, Fleming_1989, Buckdahn_SICON_2008, Li_AMO_2015, Qui_ESIAM_2013, Pham_SICON_2014, Yong_book}.

The second main objective of the paper is to establish lower and upper dynamic programming principles (DPPs) for the (lower and upper) value functions. This provides a recursive relationship of the (lower and upper) value function. Due to the law invariant property, the (lower and upper) DPPs on $[t,T] \times \mathcal{P}_2$ and the lifted (lower and upper) DPPs on $[t,T] \times \mathcal{L}_2$ are identical. For the proof, we need to consider the interaction between admissible control and nonanticipative strategies between the players. 

The third main objective of the paper is to show that (lower and upper) value functions are viscosity solutions of the associated (lower and upper) Hamilton-Jacobi-Isaacs (HJI) equations that are first-order partial differential equations (PDEs) on $[t,T] \times \mathcal{P}_2$ and $[t,T] \times \mathcal{L}_2$. Hence, unlike  \cite{Bardi_book_1997, Evans_1984, Fleming_1989, Buckdahn_SICON_2008, Li_AMO_2015, Qui_ESIAM_2013, Pham_SICON_2014}, the HJI equations of this paper are infinite-dimensional. We use the notion of derivative of a function of probability measures in $\mathcal{P}_2$ and its lifted version in $\mathcal{L}_2$ with the associated chain rule introduced in \cite{Cardaliaguet, Carmona_book_2018} to characterize the (lower and upper) HJI equations and the viscosity solution property of the (lower and upper) value functions. Furthermore, when the dynamics and running cost are independent of time, by  constructing the test function and using the law invariant property, we prove the comparison principle of viscosity solutions, which leads to uniqueness of the viscosity solution. In addition, under the Isaacs condition, the lower and upper value functions coincide. This implies that the ZSDG has a value, which is further characterized by the viscosity solution of the HJI equation.

Finally, we provide numerical examples to illustrate the theoretical results of the paper. In particular, we observe that the value of ZSDGs considered in this paper is determined by the laws (distributions) of random initial and target variables, whereas the value of classical deterministic ZSDGs is obtained by explicit values of initial and target variables.

We note that different versions of the problem treated in this paper were considered earlier in \cite{Cardaliaguet_IGTR_2008, Cosso_JMPA_2018}. However, in \cite{Cardaliaguet_IGTR_2008}, the notion of nonanticipative strategies with delay was used, which is hard to implement in practical situations. Moreover, the objective functional does not have the running cost, and the state distribution and the random target variable were not considered in \cite{Cardaliaguet_IGTR_2008}. The stochastic version of the problem of this paper was considered in \cite{Cosso_JMPA_2018}. However, the comparison principle and therefore the uniqueness of viscosity solutions were not shown. Hence, there is no guarantee that the solution of the corresponding HJI equations characterizes the value function in \cite{Cosso_JMPA_2018}. In summary, the problem formulation, the approach used, and the main results of this paper are different from those of \cite{Cardaliaguet_IGTR_2008, Cosso_JMPA_2018}.

The rest of the paper is organized as follows. Notations including the notion of derivative in $\mathcal{P}_2$ and its lifted version in $\mathcal{L}_2$, and the problem formulation are provided in Section \ref{Section_2}. The DPPs and the properties of the (lower and upper) value functions are given in Section \ref{Section_3}. The (lower and upper) HJI equations and their viscosity solutions (including existence and uniqueness) are given in Section \ref{Section_4}. Numerical examples are provided in Section \ref{Section_5}. Several potential future research problems are discussed in Section \ref{Section_6}. Five appendices include proofs of the main results.

\section{Problem Statement}\label{Section_2}
In this section, we first describe the notation used in the paper, along with some notions and properties. The precise problem formulation then follows.

\subsection{Notation}\label{Section_2_1}
The $n$-dimensional Euclidean space is denoted by $\mathbb{R}^n$, and the transpose of a vector $x \in \mathbb{R}^n$ by $x^\top$. The inner product of $x,y \in \mathbb{R}^n$ is denoted by $\langle x,y \rangle := x^\top y$, and the Euclidean norm of $x \in \mathbb{R}^n$ by $|x| := \langle x,x\rangle^{\frac{1}{2}}$. 

Let $\mathcal{C}([0,T] \times \mathbb{R}^n)$ be the set of all real-valued continuous functions defined on $[0,T] \times \mathbb{R}^n$. Let $\mathcal{C}^{1,1}([0,T]\times \mathbb{R}^n)$ be the set of real-valued functions defined on $[0,T]\times \mathbb{R}^n$ such that for $f \in \mathcal{C}^{1,1}([0,T]\times \mathbb{R}^n)$, $\partial_t f(t,x)$ (the partial derivative of $f$ with respect to $t$) and $ \partial_x f(t,x)$ (the partial derivatives of $f$ with respect to $x$) are continuous and bounded. Let $\mathcal{A}$ and $\mathcal{B}$ be Banach spaces with the norms $\|\cdot\|_{\mathcal{A}}$ and $\|\cdot\|_{\mathcal{B}}$, respectively. A function $f:\mathcal{A} \rightarrow \mathcal{B}$ is \emph{Frechet differentiable} at $x \in \mathcal{A}$ \cite[page 172]{Luenberger_book} if there exists a bounded linear operator $D_xf :\mathcal{A} \rightarrow \mathcal{B}$ such that $\lim_{\|h\|_{\mathcal{A}} \rightarrow 0} \frac{\|f(x+h)-f(x)-D_xf(x)(h)\|_{\mathcal{B}}}{\|h\|_{\mathcal{A}}} = 0$.	

Let $(\Omega, \mathcal{F}, \mathbb{P})$ be a complete probability space, and $\mathbb{E}$ be the expectation operator with respect to $\mathbb{P}$. We denote by $\mathbb{P}_x$ the distribution (or law) of a random variable $x$. Let $\mathbb{E}_{x \sim \mathbb{P}_x}$ be the expectation for which the underlying distribution (or law) is $\mathbb{P}_x$. Let $\mathcal{L}_2(\Omega,\mathbb{R}^n)$ be the set of $\mathbb{R}^n$-valued random vectors such that for $x \in \mathcal{L}_2(\Omega,\mathbb{R}^n)$, $\mathbb{E}[|x|^2] < \infty$. $\mathcal{L}_2(\Omega,\mathbb{R}^n)$ is a Hilbert space, with inner product and norm denoted by $ \mathbb{E}[\langle x,y \rangle ]$ and $\|x\|_{\mathcal{L}_2} := \mathbb{E}[|x|^2]^{1/2}$, respectively \cite{Luenberger_book, Conway_2000_book}.

Let $\mathcal{P}(\mathbb{R}^n)$ be the set of probability measures on $\mathbb{R}^n$, and $\mathcal{P}_p :=\mathcal{P}_p(\mathbb{R}^n) \subset \mathcal{P}(\mathbb{R}^n)$ be the set of probability measures with finite $p$-th moment, $p \geq 1$, i.e., for any $\mu \in \mathcal{P}_p(\mathbb{R}^n)$ with $p \geq 1$, we have $(\int_{\mathbb{R}^n} |x|^p \dd \mu(x)  )^{1/p} < \infty$. We note that $x \in \mathcal{L}_2(\Omega,\mathbb{R}^n)$ if and only if $\mu = \mathbb{P}_x \in \mathcal{P}_2(\mathbb{R}^n)$. For $x \in \mathcal{L}_2(\Omega,\mathbb{R}^n)$ with the associated law $\mu \in \mathcal{P}_2(\mathbb{R}^n)$, we can write $\mathbb{E}[x] = \int_{\mathbb{R}^n} x \dd \mu(x)$. The $p$-Wasserstein metric is defined by (see \cite[page 40]{Rachev} and \cite[Chapter 6]{Villani_book}):
\begin{align*}
W_p(\mu_1,\mu_2) &:= \bigl ( \inf_{\pi \in \Pi(\mu_1,\mu_2)} \int_{\mathbb{R}^n \times \mathbb{R}^n} |x-y|^p \dd \pi( x, y) \bigr)^{1/p},
\end{align*}
where $\mu_1,\mu_2 \in \mathcal{P}_p$, and $\Pi(\mu_1,\mu_2)$ is the collection of all probability measures on $\mathbb{R}^n \times \mathbb{R}^n$ with marginals $\mu_1$ and $\mu_2$, i.e. $\pi(A \times \mathbb{R}^n) = \mu_1(A)$ and $\pi(\mathbb{R}^n \times A) = \mu_2(A)$ for any Borel sets $A \subset \mathbb{R}^n$ \cite{Rachev}. Note that $W_2$ can equivalently be written as \cite[Chapter 6]{Villani_book}
\begin{align*}
W_2(\mu_1,\mu_2) & = \inf \bigl \{ \|x_1-x_2\|_{\mathcal{L}_2}~|~x_1,x_2 \in \mathcal{L}_2(\Omega,\mathbb{R}^n)\\
&\qquad \qquad  \text{with}~ \mathbb{P}_{x_1} = \mu_1~\text{and}~ \mathbb{P}_{x_2} = \mu_2 \bigr \}.
\end{align*}
One can easily show that $W_p$ is a metric; hence, $\mathcal{P}_p(\mathbb{R}^n)$ endowed with $W_p$, $p \geq 1$, is a metric space. For $\mu \in \mathcal{P}_2(\mathbb{R}^n)$, let $\mathcal{L}_2^\mu(\mathbb{R}^n)$ be the set of square-integrable functions with respect to $\mu$.

We next provide the notion of derivative in $\mathcal{P}_2$, and its lifted derivative in $\mathcal{L}_2$, which are introduced in \cite{Cardaliaguet, Carmona_book_2018}. Let $x \in \mathcal{L}_2(\Omega,\mathbb{R}^n)$, which implies $\mu := \mathbb{P}_x \in \mathcal{P}_2(\mathbb{R}^n)$. Let $f:\mathcal{P}_2(\mathbb{R}^n) \rightarrow \mathbb{R}$. We introduce the \emph{lifted} (extended) version of $f$, $F: \mathcal{L}_2(\Omega,\mathbb{R}^n) \rightarrow \mathbb{R}$, that is, for $x \in \mathcal{L}_2(\Omega,\mathbb{R}^n)$ (note that $\mathbb{P}_x \in \mathcal{P}_2(\mathbb{R}^n)$), $F(x) = f(\mathbb{P}_x)$. While $F$ is a function of the random variable, $f$ is a function of the distribution (law) of $x$. We say that $f$ is \emph{differentiable} at $\mathbb{P}_x \in \mathcal{P}_2(\mathbb{R}^n)$, if its lifted version $F$ is Frechet differentiable at $x \in \mathcal{L}_2(\Omega,\mathbb{R}^n)$. Let $\bar{D}_x F(x) : \mathcal{L}_2(\Omega,\mathbb{R}^n) \rightarrow \mathbb{R}$ be the corresponding Frechet derivative. Then $\bar{D}_x F(x)$ is a bounded linear functional. Since $\mathcal{L}_2(\Omega,\mathbb{R}^n)$ is a Hilbert space and its dual space is $\mathcal{L}_2(\Omega,\mathbb{R}^n)$, in view of Riesz representation theorem \cite[Theorem 3.4]{Conway_2000_book}, for any $y \in \mathcal{L}_2(\Omega,\mathbb{R}^n)$, there is a unique $D_x F(x) \in \mathcal{L}_2(\Omega,\mathbb{R}^n)$ such that $\bar{D}_x F(x)(y) = \mathbb{E}[\langle D_x F(x),y \rangle ]$. This implies that the Frechet derivative can be viewed as an element of $\mathcal{L}_2(\Omega,\mathbb{R}^n)$. In view of \cite[Theorem 6.2]{Cardaliaguet}, $D_x F(x)$ does not depend on $x$, but depends on the law (distribution) of $x$. Also, from \cite[Theorem 6.5]{Cardaliaguet}, there exists a function $\partial_{\mu} f(\mu) :\mathbb{R}^n \rightarrow \mathbb{R}^n$ with $\partial_{\mu} u(\mu) \in \mathcal{L}_2^\mu(\mathbb{R}^n)$ such that $\partial_{\mu} f(\mu) \in \mathcal{L}_2^\mu(\mathbb{R}^n)$ is a derivative of $f$ in $\mathcal{P}_2(\mathbb{R}^n)$, which can be represented as $D_x F(x) = \partial_{\mu} f(\mu)(x)$. Finally, consider the dynamical system $\dot{x}(t) = f(x(t))$ with $x(0) = x_0 \in \mathcal{L}_2(\mathbb{R}^n)$. Let $\mu_t := \mathbb{P}_{x(t)} \in \mathcal{P}_2(\mathbb{R}^n) $ be the state distribution (or law) of the dynamical system. Let $ v \in \mathcal{C}^{1,1}([0,T] \times \mathcal{P}_2(\mathbb{R}^n))$. From the notion of derivative in $\mathcal{P}_2$, the chain rule in $\mathcal{P}_2$ is $\dd v(t,\mu_t) = \partial_t v(t,\mu_t) \dd t + \int_{\mathbb{R}^n} \langle \partial_{\mu} v(t,\mu_t)(x), f(x) \rangle \dd \mu_t(x) \dd t $. Note that for the lifted chain rule in $\mathcal{L}_2$ with $V \in \mathcal{C}^{1,1}([0,T] \times \mathcal{L}_2(\Omega, \mathbb{R}^n))$, we have $\dd V(t,x(t)) = \partial_t V(t,x(t)) \dd t + \mathbb{E}[\langle D_x V(t,x(t)), f(x(t)) \rangle ]\dd t $.

\subsection{Problem Formulation}\label{Section_2_2}
Consider the dynamical system on $[t,T]$ with the initial time $t \in [0,T)$:
\begin{align}
\label{eq_1}
\dot{x}(s) := \frac{\dd x(s)}{\dd s}  = f(s,x(s),\mathbb{P}_{s}^{t,\nu_x;u,v},u(s),v(s)),	
\end{align}
where $x \in \mathbb{R}^{n}$ is the state with the random initial condition $x(t) = x_0 \in \mathcal{L}_2(\Omega,\mathbb{R}^n)$ having the law $\nu_x := \mathbb{P}_{x_0} \in \mathcal{P}_2(\mathbb{R}^n)$, $u \in U \subset \mathbb{R}^{m_1}$ is the control of Player 1, and $v \in V \subset \mathbb{R}^{m_2}$ is the control of Player 2. We assume that $U$ and $V$ are compact. The set of admissible controls for Player 1, $\mathcal{U}[t,T]$, is defined such that for $u \in \mathcal{U}[t,T]$, $u:[t,T] \rightarrow U$ is a measurable function. The set of admissible controls for Player 2, $\mathcal{V}[t,T]$, is defined in a similarly way. Let $\mathcal{U} := \mathcal{U}[0,T]$ and $\mathcal{V}:=\mathcal{V}[0,T]$. In (\ref{eq_1}), $\mathbb{P}_{s}^{t,\nu_x;u,v} \in \mathcal{P}_2(\mathbb{R}^{n})$ denotes the law (equivalently distribution) of the dynamical system at time $s$ that is dependent on the law of the initial condition $\nu_x$ (see Remark \ref{Remark_A_1} in Appendix \ref{Appendix_A}), as well as $u$ and $v$. We introduce the following assumption:
\begin{enumerate}
\setlength{\itemindent}{1.5em}
\item[(H.1)] $f:[0,T] \times \mathbb{R}^n \times \mathcal{P}(\mathbb{R}^n)  \times U \times V  \rightarrow \mathbb{R}^n$ is bounded, where $f$ is continuous in $(t,u,v)$ for each $\nu \in \mathcal{P}(\mathbb{R}^n)$, and satisfies the Lipschitz condition: for $t \in [0,T]$, $u \in U$, $v \in V$, $x_1,x_1 \in \mathbb{R}^n$ and $\nu_1,\nu_2 \in \mathcal{P}(\mathbb{R}^n)$ and for $K>0$, it holds that $|f(t,x_1,\nu_1,u,v) - 	f(t,x_2,\nu_2,u,v) | \leq K (|x_1 - x_2| + W_1(\nu_1,\nu_2) )$.
\end{enumerate}
Then, for $x_0 \in \mathbb{R}^n$, (\ref{eq_1}) admits a unique solution on $[0,T]$ \cite{Rene_anals_2015,Jourdain_2008}.
%\footnote{The proof of this result has not been reported in the literature, but is similar to the standard fixed-point approaches in \cite{Khalil_book,Rene_anals_2015,Jourdain_2008}.}

Let $z \in \mathcal{L}_2(\Omega,\mathbb{R}^n)$ with the law $\nu_z := \mathbb{P}_z \in \mathcal{P}_2(\mathbb{R}^n)$, which is independent of $x_0$. Here, $z$ is the \emph{target variable} in the objective functional (see (\ref{eq_2}) and Remark \ref{Remark_2}). Let $y := (x_0,z) \in \mathcal{L}_2(\Omega,\mathbb{R}^{n}) \times \mathcal{L}_2(\Omega,\mathbb{R}^{n}) =: (\mathcal{L}_2(\Omega,\mathbb{R}^{n}))^2$ with the law given by $\nu := (\nu_x,\nu_z) = \mathbb{P}_{y} \in \mathcal{P}_2(\mathbb{R}^{n}) \times \mathcal{P}_2(\mathbb{R}^{n}) =: (\mathcal{P}_2(\mathbb{R}^{n}))^2$. The objective functional for the two-player ZSDG of this paper is then given by
\begin{align}
\label{eq_2}
& J(t,\nu_x,\nu_z; u,v) =: J(t,\nu; u,v)   \\
& \equiv J(t,x_0,z;u,v) =: J(t,y;u,v)   \nonumber \\
& = \mathbb{E}_{(x_0,z) \sim \nu} \Bigl [ \int_t^T l(s,x(s),\mathbb{P}_{s}^{t,\nu_x;u,v}, u(s),v(s)) \dd s + m(x(T),z) \Bigr ], \nonumber
\end{align}
where $J$ is cost to Player 1 (minimizer) and payoff to Player 2 (maximizer). The notation in the first line of (\ref{eq_2}) indicates that $J$ is defined on $[0,T] \times (\mathcal{P}_2(\mathbb{R}^{n}))^2 \times \mathcal{U} \times \mathcal{V}$, whereas the notation in the second line of (\ref{eq_2}) stands for $J$ as a functional on $[0,T] \times (\mathcal{L}_2(\Omega,\mathbb{R}^{n}))^2 \times \mathcal{U} \times \mathcal{V}$, and the two are equivalent because of the correspondence between $\mathcal{L}_2$ and $\mathcal{P}_2$ discussed in Section \ref{Section_2_1}. Let $J(u,v)  := J(0,\nu_x,\nu_z; u,v) = J(0,x_0,z; u,v)$. We have the following assumption:

\begin{enumerate}
\setlength{\itemindent}{1.5em}
\item[(H.2)] $l:[0,T] \times \mathbb{R}^n \times \mathcal{P}(\mathbb{R}^n) \times U \times V \rightarrow \mathbb{R}$ is bounded, where $l$ is continuous in $(t,u,v)$ for each $\nu \in \mathcal{P}(\mathbb{R}^n)$ and satisfies the Lipschitz condition: for $t \in [0,T]$, $u \in U$, $v \in V$, $x_1,x_2 \in \mathbb{R}^n$ and $\nu_1,\nu_2 \in \mathcal{P}(\mathbb{R}^n)$ and for $K > 0$, it holds that $|l(t,x_1,\nu_1, u,v) - l(t,x_2,\nu_2, u,v)| \leq K(|x_1 - x_2| + W_1(\nu_1,\nu_2))$. Also, $m : \mathbb{R}^n \times \mathbb{R}^n \rightarrow \mathbb{R}$ is bounded, which is Lipschitz continuous in $(x,z)$ with Lipschitz constant $K > 0$. 
\end{enumerate}
\begin{Remark}\label{Remark_2}
The random variable $z$ included in the terminal cost $m$ of (\ref{eq_2}) is called the \emph{target variable}, which captures the constraint of the state process at the terminal time. Specifically, given $z$, $m$ can be used such that the distance between the law of the state process and the target distribution (the law of $z$) can be optimized via the control processes $u$ and $v$.
\end{Remark}

We next introduce the notion of nonanticipative strategies for Player 1 and Player 2; see also \cite{Bardi_book_1997, Evans_1984, Fleming_1989, Buckdahn_SICON_2008, YU_SICON_2015}.
\begin{Definition}\label{Definition_1}
	A strategy for Player 1 is a mapping $\alpha:\mathcal{V} \rightarrow \mathcal{U}$. A strategy for Player 1 is \emph{nonanticipative} if for any $s \in [t,T]$, and $v_1,v_2 \in \mathcal{V}$, $v_1(\bar{s}) = v_2(\bar{s})$ for $\bar{s} \in [t,s]$ implies that $\alpha(v_1)(\bar{s}) = \alpha(v_2)(\bar{s})$ for $\bar{s} \in [t,s]$, this being true for each $t\in [0, T]$. The set of nonanticipative strategies for Player 1 on $[t,T]$ is denoted by $\mathcal{A}[t,T]$. Let $\mathcal{A} := \mathcal{A}[0,T]$. A strategy for Player 2 is a measurable mapping $\beta:\mathcal{U} \rightarrow \mathcal{V}$. A nonanticipative strategy for Player 2 is defined in a similar way as Player 1's. The set of nonanticipative strategies for Player 2 on $[t,T]$ is denoted by $\mathcal{B}[t,T]$. Let $\mathcal{B} :=\mathcal{B}[0,T]$.
\end{Definition}

Using Definition \ref{Definition_1}, for $t \in [0,T]$ and $\nu = (\nu_x,\nu_z)  \in (\mathcal{P}_2(\mathbb{R}^{n}))^2$, the lower value function for (\ref{eq_2}) is defined by $L:[0,T] \times (\mathcal{P}_2(\mathbb{R}^{n}))^2\rightarrow \mathbb{R}$ with 
\begin{align}	
\label{eq_3}
L(t,\nu_x,\nu_z) & = L(t,\nu)  := \inf_{\alpha \in \mathcal{A}[t,T]} \sup_{v \in \mathcal{V}[t,T]} J(t,\nu;\alpha(v),v),
\end{align}
and the upper value function is defined by $M:[0,T] \times (\mathcal{P}_2(\mathbb{R}^{n}))^2 \rightarrow \mathbb{R}$ with 
\begin{align}	
\label{eq_4}
M(t,\nu_x,\nu_z) & = M(t,\nu)   := \sup_{\beta \in \mathcal{B}[t,T]} \inf_{u \in \mathcal{U}[t,T]}J(t,\nu;u,\beta(u)).
\end{align}
Note that $L(T,\nu) = M(T,\nu) = \int_{\mathbb{R}^n \times \mathbb{R}^n} m(x,z) \dd \nu(x,z)$. Unlike the deterministic case, (\ref{eq_3}) and (\ref{eq_4}) are parametrized by the initial time, the initial distribution (law) of (\ref{eq_1}), and the target distribution in (\ref{eq_2}).

For $t \in [0,T]$ and $y = (x,z) \in (\mathcal{L}_2(\Omega,\mathbb{R}^{n}))^2$, define the \emph{lifted} lower value function, $\mathbb{L}:[0,T] \times  (\mathcal{L}_2(\Omega,\mathbb{R}^{n}))^2 \rightarrow \mathbb{R}$ with 
\begin{align}
\label{eq_6}
\mathbb{L}(t,x,z) &= \mathbb{L}(t,y) := \inf_{\alpha \in \mathcal{A}[t,T]} \sup_{v \in \mathcal{V}[t,T]} J(t,y;\alpha(v),v), 
\end{align}
and the lifted upper value function, $\mathbb{M}:[0,T] \times  (\mathcal{L}_2(\Omega,\mathbb{R}^{n}))^2 \rightarrow \mathbb{R}$ with
\begin{align}
\label{eq_7}
\mathbb{M}(t,x,z)  =	\mathbb{M}(t,y)  := \sup_{\beta \in \mathcal{B}[t,T]} \inf_{u \in \mathcal{U}[t,T]}J(t,y;u,\beta(u)).
\end{align}
Note that $\mathbb{L}(T,y) = \mathbb{M}(T,y) = \mathbb{E}_{y \sim \nu}[m(x,z)]$. As mentioned in Section \ref{Section_2_1}, the lifted value functions depend on only the law of $y=(x,z)$.\footnote{The value functions and their lifted versions are defined based on the notation in (\ref{eq_2}). Note that the value functions and their corresponding lifted versions are deterministic and identical, where the detailed proof of the latter is given in Lemma \ref{Lemma_1}. As stated in Section \ref{Section_2_1} and \cite{Cardaliaguet, Carmona_book_2018}, the motivation for introducing the lifted value functions on $\mathcal{L}_2$ is to utilize the notion of derivative in $\mathcal{L}_2$, which allows us to characterize the explicit derivative of the (inverse-lifted) value function on $\mathcal{P}_2$. Note that $\mathcal{L}_2$ is a Hilbert space, but $\mathcal{P}_2$ with the Wasserstein metric is not. 
}

\begin{Remark}\label{Remark_3}
Unlike ZSDGs in \cite{Bardi_book_1997, Evans_1984, Fleming_1989, Buckdahn_SICON_2008, Li_AMO_2015, Qui_ESIAM_2013, Pham_SICON_2014}, the lower and upper value functions in (\ref{eq_3})-(\ref{eq_7}) are defined on infinite-dimensional spaces $\mathcal{P}_2$ and $\mathcal{L}_2$.
\end{Remark}

Before concluding this section, we provide a few examples of ZSDGs that fit into the framework laid out above.

\begin{Example}\label{Example_1}
	The state dynamics in (\ref{eq_1}) can be regarded as a McKean-Vlasov dynamics, where the evolution of the state process depends on its distribution. This is closely related to \emph{mean-field games} and \emph{mean-field type control}, which have been studied extensively in the literature, particularly, for reducing variation of random effects on the controlled process and macroscopic analysis of large-scale interacting multi-agent systems \cite{Lasry, Carmona_book_2018, Andersson_AMO_2010, 7047683, Cardaliaguet, Rene_SICON_2013, Yong_SICON_2013_MF, Bensoussan_Arxiv_2014, Rene_anals_2015, Cardaliaguet_MFE_2018, Djehiche_AMO_2018, Averboukh_DGAA_2018, Moon_DGAA_2019}. For example, we may take $f(t,x(t),\mathbb{E}[x(t)], u(t), v(t))$, $l(t,x(t),\mathbb{E}[x(t)], u(t), v(t))$ and $m(x(T),\mathbb{E}[x(T)])$ to optimize the objective functional under the mean-field effect. Notice that if $l \equiv 0$ and $J(u,v) = \mathbb{E}[x^2(T)] - (\mathbb{E}[x(T)])^2$, then what we have is a class of mean-variance optimization problems. 
\end{Example}

%Hajek_SLT_2019
\begin{Example}\label{Example_2}
In statistical learning theory and its applications, we often need to optimize the worst-case	empirical criterion (or risk) \cite{Goodfellow_book}. Specifically, assume that $(x_{0}^i, z^i)$, $i=1,\ldots,N$, is an i.i.d. random pair sampled according to $\nu = (\nu_x,\nu_z) \in (\mathcal{P}_2(\mathbb{R}^{n}))^2$. Consider
\begin{align*}
J^N(u,v) & = \frac{1}{N} \sum_{i=1}^N \Bigl [ \int_0^T l(t,x^{(i)}(t), u(t),v(t)) \dd t  + m (x^{(i)}(T),z^{(i)}) \Bigr ] \\
\dot{x}^{(i)}(t) &= f(t,x^{(i)}(t),u,v),~ x^{(i)}(0) = x_0^{(i)},~ i=1,\ldots,N.
\end{align*}
Note that we have $\lim_{N \rightarrow \infty} \frac{1}{N} \sum_{i=1}^N x^{i}(t) = \mathbb{E}[x^{i}(t)]$ almost surely in view of the law of large numbers, which, together with (H.1) and (H.2) implies that $\lim_{N \rightarrow \infty} J^N(u,v) = J(u,v)$. Hence, from the minimization point of view, the class of ZSDGs of this paper can be viewed as worst-case empirical optimization when the sample size $N$ is arbitrarily large.
\end{Example}
\begin{Example}\label{Example_3}
Consider the two adversarial vehicles model:
\begin{align*}
J(u,v) & = \mathbb{E}_{(x_0,z) \sim \nu} \bigl [ | x(T) - z | \bigr ]	 \\
\dot{x}(t) &:= \begin{bmatrix}
\dot{x}_1(t) \\
\dot{x}_2(t) \\
\dot{x}_3(t)	
\end{bmatrix} = \begin{bmatrix}
 	-v_a + v_b \cos x_3(t) + v(t) x_2(t) \\
 	v_b \sin x_3(t) - v(t) x_1(t) \\
 	u(t)-v(t)
 \end{bmatrix},
\end{align*}
where $x(0) = x_0$ is the random initial condition, $z$ is the random target variable, $v_a, v_b > 0$ are constants of the two vehicles and $u,v$ are velocities of the two vehicles \cite{Mitchell_TAC_2005}. When specialized to this setting, the class of ZSDGs in this paper can be seen as a pursuit-evasion game of two vehicles with random initial and target pair. Its deterministic version with different settings and applications to characterization of reachable sets were studied in \cite{Buckdahn_Cardaliaguet_DGAA_2011, Basar_Zaccour_book, Mitchell_TAC_2005, Margellos_TAC_2011}, and the references therein.
\end{Example}

\section{Dynamic Programming Principles}\label{Section_3}

In this section, we obtain the dynamic programming principles (DPPs) for the lower and upper value functions. 

We first provide some properties of the value functions. The following lemma shows that the value functions are law invariant. The proof can be found in Appendix \ref{Appendix_A}.

\begin{lemma}\label{Lemma_1}
Suppose that (H.1) and (H.2) hold. Then we have $L(t,\nu) = \mathbb{L}(t,y)$ and $M(t,\nu) = \mathbb{M}(t,y)$ for any $y = (x,z) \in (\mathcal{L}_2(\Omega,\mathbb{R}^{n}))^2$ with the law of $y$ being $\nu = (\nu_x,\nu_z) \in (\mathcal{P}_2(\mathbb{R}^{n}))^2$, i.e., $\nu  := \mathbb{P}_{y} = \mathbb{P}_{(x,z)}\in (\mathcal{P}_2(\mathbb{R}^{n}))^2$.
\end{lemma}

The next lemma shows the continuity of the value functions. The proof is provided in Appendix \ref{Appendix_A}.
\begin{lemma}\label{Lemma_2}
Suppose that (H.1) and (H.2) hold. Then the lifted value functions in (\ref{eq_6}) and (\ref{eq_7}) are continuous in $(t,y) \in [0,T] \times (\mathcal{L}_2(\Omega,\mathbb{R}^{n}))^2$. Furthermore, the value functions in (\ref{eq_3}) and (\ref{eq_4}) are continuous in $(t,\nu) \in [0,T] \times (\mathcal{P}_{2}(\mathbb{R}^{n}))^2$.
\end{lemma}

We now obtain in the following theorem the DPPs, whose proof is given in Appendix \ref{Appendix_B}.

\begin{theorem}\label{Theorem_1}
Suppose that (H.1) and (H.2) hold. Then for any $y = (x,z) \in (\mathcal{L}_2(\Omega,\mathbb{R}^{n}))^2$ with $\nu := (\nu_x,\nu_z) = \mathbb{P}_y \in (\mathcal{P}_2(\mathbb{R}^{n}))^2$ and $t,t+\tau \in [0,T]$ with $t < t + \tau$, the lifted lower and upper value functions  in (\ref{eq_6}) and (\ref{eq_7}), respectively, satisfy the following DPPs:
\begin{align}
\label{eq_20}
 \mathbb{L}(t,x,z) & = 	\inf_{\alpha \in \mathcal{A}[t,t+\tau]} \sup_{v \in \mathcal{V}[t,t+\tau]} \mathbb{E}_{(x,z)\sim \nu} \Bigl [ \int_t^{t+\tau} l(s,x_s^{t,x,\nu_x;\alpha(v),v}, \\
&\qquad \mathbb{P}_{s}^{t,\nu_x;\alpha(v),v}, \alpha(v)(s), v(s)) \dd s  + \mathbb{L}(t+\tau,x_{t+\tau}^{t,x,\nu_x;\alpha(v),v},z) \Bigr ] \nonumber \\
\label{eq_21}
\mathbb{M}(t,x,z) & = \sup_{\beta \in \mathcal{B}[t,t+\tau]} \inf_{u \in \mathcal{U}[t,t+\tau]} \mathbb{E}_{(x,z)\sim \nu} \Bigl [ \int_t^{t+\tau} l(s,x_s^{t,x,\nu_x;u,\beta(u)}, \\
& \qquad \mathbb{P}_{s}^{t,\nu_x; u,\beta(u)}, u(s), \beta(u)(s)) \dd s + \mathbb{M}(t+\tau,x_{t+\tau}^{t,x,\nu_x;u,\beta(u)},z) \Bigr ]. \nonumber
\end{align}
Equivalently, for any $y = (x,z) \in (\mathcal{L}_2(\Omega,\mathbb{R}^{n}))^2$ with the law $\nu := (\nu_x,\nu_z) = \mathbb{P}_y \in (\mathcal{P}_2(\mathbb{R}^{n}))^2$ and $t,t+\tau \in [0,T]$ with $t < t + \tau$, the lower and upper value functions in (\ref{eq_3}) and (\ref{eq_4}), respectively, satisfy the following DPPs:
\begin{align}
\label{eq_22}
& L(t,\nu_x,\nu_z)  \\
& = \inf_{\alpha \in \mathcal{A}[t,t+\tau]} \sup_{v \in \mathcal{V}[t,t+\tau]} \Bigl \{  \int_t^{t+\tau} \int_{\mathbb{R}^{n}} l(s,x_{s}^{t,x,\nu_x;\alpha(v),v}, \mathbb{P}_{s}^{t,\nu_x; \alpha(v),v},\nonumber \\
&\qquad \qquad \alpha(v)(s),v(s)) \dd \mathbb{P}_{s}^{t,\nu_x;\alpha(v),v} (x) \dd s  + L(t+\tau,\mathbb{P}_{t+\tau}^{t,\nu_x;\alpha(v),v},\nu_z) \Bigr \} \nonumber \\
\label{eq_23}
& M(t,\nu_x,\nu_z)  \\
& = \sup_{\beta \in \mathcal{B}[t,t+\tau]} \inf_{u \in \mathcal{U}[t,t+\tau]} \Bigl \{  \int_t^{t+\tau} \int_{\mathbb{R}^{n}} l(s,x_{s}^{t,x,\nu_x;u,\beta(u)},  \mathbb{P}_{s}^{t,\nu_x;u,\beta(u)}, \nonumber \\
&\qquad \qquad u(s),\beta(u)(s)) \dd \mathbb{P}_{s}^{t,\nu_x;u,\beta(u)} (x) \dd s + M(t+\tau,\mathbb{P}_{t+\tau}^{t,\nu_x;u,\beta(u)}, \nu_z) \Bigr \}. \nonumber
\end{align}
\end{theorem}

\section{HJI Equations and Viscosity Solutions} \label{Section_4}

In this section, we address the issue of the lower and upper value functions being unique viscosity solutions of the associated Hamilton-Jacobi-Isaacs (HJI) equations, which are first-order partial differential equations defined on infinite-dimensional spaces, particularly $\mathcal{P}_2$ and $\mathcal{L}_2$.
 
The lower HJI equation on $[0,T] \times (\mathcal{P}_2(\mathbb{R}^{n}))^2$ is given by
\begin{align}
\label{eq_32}
\begin{cases}
\partial_t L(t,\nu_x,\nu_z) + H^{-}(t,\nu,\partial_{\nu_x}L(t,\nu_x,\nu_z)) = 0 \\
L(T,\nu_x,\nu_z) = \int_{\mathbb{R}^{n} \times \mathbb{R}^n} m(x,z) \dd \nu(x,z)
\end{cases},
\end{align}
and the upper HJI equation on $[0,T] \times (\mathcal{P}_2(\mathbb{R}^{n}))^2$ is as follows:
\begin{align}
\label{eq_33}
\begin{cases}
\partial_t M(t,\nu_x,\nu_z) + H^{+}(t,\nu,\partial_{\nu_x}M(t,\nu_x,\nu_z)) = 0 \\
M(T,\nu_x,\nu_z) = \int_{\mathbb{R}^{n} \times \mathbb{R}^n} m(x,z) \dd \nu(x,z)
\end{cases},	
\end{align}
where $H^{-},H^{+}:[0,T] \times (\mathcal{P}_2(\mathbb{R}^{n}))^2 \times \mathcal{L}_{2}^{\nu_x}(\mathbb{R}^{n}) \rightarrow \mathbb{R}$ are the Hamiltonians:
\begin{align}
\label{eq_34}
H^{-}(t,\nu,p) & := \sup_{v \in V} \inf_{u \in U} \Bigl \{ \int_{\mathbb{R}^{n}} [ \langle p,f(t,x,\nu_x,u,v) \rangle   + l(t,x,\nu_x,u,v) ] \dd \nu_x(x)\Bigr \} \\
H^{+}(t,\nu,p) & := \inf_{u \in U} \sup_{v \in V}  \Bigl \{ \int_{\mathbb{R}^{n}} [ \langle p,f(t,x,\nu_x,u,v) \rangle   + l(t,x,\nu_x,u,v) ] \dd \nu_x(x)\Bigr \}.  \nonumber
\end{align}

Viscosity solutions to (\ref{eq_32}) and (\ref{eq_33}) are defined as follows; see also \cite{Crandall_Lions_1983, Buckdahn_SICON_2008, Bardi_book_1997, Fleming_1989, Evans_1984, Li_AMO_2015, Yong_book} and the references therein:

\begin{Definition}\label{Definition_2}
\begin{enumerate}[(i)]
\item A real-valued function $L \in \mathcal{C}([0,T] \times (\mathcal{P}_2(\mathbb{R}^{n}))^2)$ is said to be a viscosity subsolution (resp. supersolution) of the lower HJI equation in (\ref{eq_32}) if $L(T,\nu) \leq \int_{\mathbb{R}^{n} \times \mathbb{R}^n} m(x,z) \dd \nu(x,z)$ (resp. $L(T,\nu) \geq \int_{\mathbb{R}^{n} \times \mathbb{R}^n} m(x,z) \dd \nu(x,z)$) for $\nu \in (\mathcal{P}(\mathbb{R}^{n}))^2$, and if further for all test functions $\phi \in \mathcal{C}^{1,1}([0,T] \times (\mathcal{P}_2(\mathbb{R}^{n}))^2)$ and $(t,\nu) \in [0,T) \times (\mathcal{P}_2(\mathbb{R}^{n}))^2$, the following inequality holds at the local maximum (resp. local minimum) point $(t,\nu)$ of $L-\phi$:
%satisfying $(L-\phi)(t,\nu) = 0$:
	\begin{align*}
	& \partial_t \phi(t,\nu) + H^{-}(t,\nu,\partial_{\nu_x}\phi(t,\nu)) \geq 0 \\
	(\text{resp.}~ & \partial_t \phi(t,\nu) + H^{-}(t,\nu,\partial_{\nu_x}\phi(t,\nu)) \leq 0).
	\end{align*}
%	\item A real-valued function $L \in \mathcal{C}([0,T] \times (\mathcal{P}_2(\mathbb{R}^{n}))^2)$ is said to be a viscosity subsolution (resp. supersolution) of the lower HJI equation in (\ref{eq_32}) if $L(T,\nu) \leq \int_{\mathbb{R}^{n} \times \mathbb{R}^n} m(x,z) \dd \nu(x,z)$ (resp. $L(T,\nu) \geq \int_{\mathbb{R}^{n} \times \mathbb{R}^n} m(x,z) \dd \nu(x,z)$) for any $\nu \in (\mathcal{P}(\mathbb{R}^{n}))^2$, and if for all test functions $\phi \in \mathcal{C}^{1,1}([0,T] \times (\mathcal{P}_2(\mathbb{R}^{n}))^2)$ satisfying 
%	\begin{align*}
%	& 0 = (L-\phi)(t,\nu) = \sup_{t^\prime \in [0,T],\nu^\prime \in (\mathcal{P}_2(\mathbb{R}^{n}))^2} (L-\phi)(t^\prime,\nu^\prime), \\
%	(\text{resp.}~ 	& 0 = (L-\phi)(t,\nu) = \inf_{t^\prime \in [0,T],\nu^\prime \in (\mathcal{P}_2(\mathbb{R}^{n}))^2} (L-\phi)(t^\prime,\nu^\prime)), 
%	\end{align*}
%	the following inequality holds:
%	\begin{align*}
%	& \partial_t \phi(t,\nu) + H^{-}(t,\nu,\partial_{\nu_x}\phi(t,\nu)) \geq 0	 \\
%	(\text{resp.}~ & \partial_t \phi(t,\nu) + H^{-}(t,\nu,\partial_{\nu_x}\phi(t,\nu)) \leq 0).
%	\end{align*}
	\item A real-valued function $L \in \mathcal{C}([0,T] \times (\mathcal{P}_2(\mathbb{R}^{n}))^2)$ is said to be a viscosity solution of (\ref{eq_32}) if it is both a viscosity subsolution and a viscosity supersolution. The viscosity subsolution, supersolution, and solution of the HJI equation in (\ref{eq_33}) are defined in similar ways.
\end{enumerate}
\end{Definition}

The following theorem, whose proof is given in Appendix C, now establishes the viscosity solution property of the value functions in (\ref{eq_3}) and (\ref{eq_4}). 

\begin{theorem}\label{Theorem_2}
Suppose that (H.1) and (H.2) hold. Then, the lower value function $L$ is a viscosity solution to the lower HJI equation (\ref{eq_32}). The upper value function $M$ is a viscosity solution to the upper HJI equation (\ref{eq_33}).
\end{theorem}

With the lifted value functions, the lifted lower and upper HJI equations are given by
\begin{align}
\label{eq_36}
& \begin{cases}
\partial_t \mathbb{L}(t,x,z) + \mathbb{H}^{-}(t,y,D_x\mathbb{L}(t,x,z)) = 0 \\
\mathbb{L}(T,x,z) = \mathbb{E}[m(x,z)] 
\end{cases}		\\
%\end{align}
%and the upper HJI equation is as follows:
%\begin{align}
\label{eq_37}
& \begin{cases}
\partial_t \mathbb{M}(t,x,z) + \mathbb{H}^{+}(t,y,D_x\mathbb{M}(t,x,z)) = 0 \\
\mathbb{M}(T,x,z) = \mathbb{E}[m(x,z)]
\end{cases},
\end{align}
where $\mathbb{H}^{-},\mathbb{H}^{+}:[0,T] \times (\mathcal{L}_2(\Omega,\mathbb{R}^{n}))^2 \times \mathcal{L}_2(\Omega,\mathbb{R}^{n}) \rightarrow \mathbb{R}$ are the (lifted) Hamiltonians defined by
\begin{align}
\label{eq_38}
	\mathbb{H}^{-}(t,y,p) & := \sup_{v \in V} \inf_{u \in U} \mathbb{E} \bigl [ \langle p,f(t,x,\nu_x,u,v) \rangle  + l(t,x,\nu_x,u,v) \bigr ]  \\
	\mathbb{H}^{+}(t,y,p) & :=  \inf_{u \in U} \sup_{v \in V} \mathbb{E} \bigl [ \langle p,f(t,x,\nu_x,u,v) \rangle  + l(t,x,\nu_x,u,v) \bigr ]. \nonumber
\end{align}
See Section \ref{Section_2_1} for the notion of derivative in $\mathcal{L}_2$ and its relationship with the derivative in $\mathcal{P}_2$. As stated in Section \ref{Section_2_2}, from the definition of the value functions (\ref{eq_6}) and (\ref{eq_7}), the lifted HJI equations  (\ref{eq_36}) and (\ref{eq_37}) are dependent on the law of $(x,z)$.

\begin{Remark}\label{Remark_4}
\begin{enumerate}[(i)]
\item 	From Remark \ref{Remark_3}, the HJI equations in (\ref{eq_32}), (\ref{eq_33}), (\ref{eq_36}) and (\ref{eq_37}) are defined on infinite-dimensional spaces.
\item The definition of the viscosity solution for the lifted HJI equations in (\ref{eq_36}) and (\ref{eq_37}) (the solution belongs to $\mathcal{C}([0,T] \times (\mathcal{L}_2(\Omega, \mathbb{R}^{n}))^2)$) is identical with Definition \ref{Definition_2}, except that we need to use the test function $\phi \in \mathcal{C}^{1,1}([0,T] \times (\mathcal{L}_2(\Omega, \mathbb{R}^{n}))^2)$ and (\ref{eq_36})-(\ref{eq_38}) instead of (\ref{eq_32})-(\ref{eq_34}).
\end{enumerate}

\end{Remark}

%\begin{Definition}\label{Definition_3}
%\begin{enumerate}[(i)]
%	\item A real-valued function $\mathbb{L} \in \mathcal{C}([0,T] \times \mathcal{L}_2(\Omega, \mathbb{R}^{2n}))$ is said to be a viscosity subsolution (resp. supersolution) to (\ref{eq_36}) if $\mathbb{L}(T,x,z) \leq \mathbb{E}[m(x,z)] $ (resp. $\mathbb{L}(T,x,z) \geq  \mathbb{E}[m(x,z)]$) for $y = (x,z) \in \mathcal{L}_2(\Omega, \mathbb{R}^{2n})$, and if for all test functions $\phi \in \mathcal{C}^{1,1}([0,T] \times \mathcal{L}_2(\Omega, \mathbb{R}^{2n}))$ and for all $(t,y) \in [0,T) \times \mathcal{L}_2(\Omega, \mathbb{R}^{2n})$, the following inequality holds at the local maximum (resp. local minimum) point  $(t,y)$ of $\mathbb{L}-\phi$:
%	\begin{align*}
%	& \partial_t \phi(t,y) + \mathbb{H}^{-}(t,y,D_x \phi(t,x,z)) \geq 0 \\
%	(\text{resp.}~ & \partial_t \phi(t,y) + \mathbb{H}^{-}(t,y,D_x \phi(t,x,z))  \leq 0).
%	\end{align*}
%\item A real-valued function $\mathbb{L} \in \mathcal{C}([0,T] \times \mathcal{L}_2(\Omega, \mathbb{R}^{2n}))$ is a viscosity solution to (\ref{eq_36}) if it is both a viscosity subsolution and supersolution.
%\item The viscosity subsolution, supersolution, and solution of the HJI equation in (\ref{eq_37}) are defined in similar ways.
%\end{enumerate}
%% with $\mathbb{H}^{+}$ (\ref{eq_39}) 
%\end{Definition}

We have the following result, whose proof is similar to that of Theorem \ref{Theorem_2}.
\begin{proposition}\label{Proposition_1}
Suppose that (H.1) and (H.2) hold. Then, the (lifted) lower value function $\mathbb{L}$ is a viscosity solution to the (lifted) lower HJI equation in (\ref{eq_36}). The (lifted) upper value function $\mathbb{M}$ is a viscosity solution to the (lifted) upper HJI equation in (\ref{eq_37}).	
\end{proposition}

Next, we state the comparison results of the viscosity solutions in Theorem \ref{Theorem_2} and Proposition \ref{Proposition_1}, with the proofs relegated to Appendix \ref{Appendix_D}. We need the following assumption:

\begin{enumerate}
\setlength{\itemindent}{1.5em}
	\item[(H.3)] $f$ and $l$ are independent of $t$. 
\end{enumerate}

\begin{theorem}\label{Theorem_3}
Assume that (H.1)-(H.3) hold. Then:
\begin{enumerate}[(i)] 
\item Suppose that $L_1$ (resp. $M_1$) and $L_2$ (resp. $M_2$) are bounded and Lipschitz continuous viscosity subsolution and viscosity supersolution of (\ref{eq_32}) (resp. (\ref{eq_33})), respectively. Then, the following result holds:
\begin{align}
\label{eq_43}
\begin{cases}
L_1(t,\nu)  \leq L_2(t,\nu),~ \forall (t,\nu) \in [0,T] \times (\mathcal{P}_2(\mathbb{R}^{n}))^2	\\
M_1(t,\nu)  \leq M_2(t,\nu),~ \forall (t,\nu) \in [0,T] \times (\mathcal{P}_2(\mathbb{R}^{n}))^2
\end{cases}.
\end{align}
\item Suppose that $\mathbb{L}_1$ and  (resp. $\mathbb{M}_1$) and $\mathbb{L}_2$ (resp. $\mathbb{M}_2$) are bounded and Lipschitz continuous viscosity subsolution and viscosity supersolution of (\ref{eq_36}) (resp. (\ref{eq_37})), respectively. Then, the following result holds:
\begin{align}
\label{eq_44}
\begin{cases}
\mathbb{L}_1(t,y)  \leq \mathbb{L}_2(t,y),~ \forall (t,y) \in [0,T] \times (\mathcal{L}_2(\Omega, \mathbb{R}^{n}))^2	\\
\mathbb{M}_1(t,y)  \leq \mathbb{M}_2(t,y),~ \forall (t,y) \in [0,T] \times (\mathcal{L}_2(\Omega, \mathbb{R}^{n}))^2
\end{cases}.
\end{align}
\end{enumerate}
\end{theorem}

Based on Theorem \ref{Theorem_3}, we have the following uniqueness result. The proof is given in Appendix \ref{Appendix_E}.
\begin{corollary}\label{Corollary_1} 
Suppose that (H.1) and (H.2) hold, and that the (lower and upper) value functions are bounded. Then, the lower (resp. upper) value function in (\ref{eq_3}) (resp. (\ref{eq_4})) is the unique viscosity solution to the lower (resp. upper) HJI equation in (\ref{eq_32}) (resp. (\ref{eq_33})). Also, the lifted lower (resp. upper) value function in (\ref{eq_6}) (resp. (\ref{eq_7})) is the unique viscosity solution to the lifted lower (resp. upper) HJI equation in (\ref{eq_36}) (resp. (\ref{eq_37})).
\end{corollary}

\begin{Remark}\label{Remark_5}
In view of Corollary \ref{Corollary_1}, by solving the lower (resp. upper) HJI equation in (\ref{eq_32}) or (\ref{eq_36}) (resp. (\ref{eq_33}) or (\ref{eq_37})), we can characterize the lower (resp. upper) value function of the ZSDG of this paper.
\end{Remark}

To proceed further, we now introduce the \emph{Isaacs conditions}:
\begin{align}
\begin{cases}
\label{eq_50}
H^{-}(t,\nu,p) = 	H^{+}(t,\nu,p) \\
%\forall (t,\nu,p) \in [0,T] \times (\mathcal{P}_2(\mathbb{R}^{n}))^2 \times \mathcal{L}_{2}^{\nu_x}(\mathbb{R}^{n}) \\
\mathbb{H}^{-}(t,y,p) = \mathbb{H}^{+}(t,y,p)
%~ \forall (t,y,p) \in \mathbb{H}^{+}:[0,T] \times (\mathcal{L}_2(\Omega,\mathbb{R}^{n}))^2 \times \mathcal{L}_2(\Omega,\mathbb{R}^{n})
\end{cases}.
\end{align}
Note that due to the law invariant property, the conditions in (\ref{eq_50}) are equivalent. Then under the Isaacs condition, we have the following result, whose proof can be found in Appendix \ref{Appendix_E}.

\begin{corollary}\label{Corollary_2}
	Suppose that (H.1)-(H.3) and (\ref{eq_50}) hold. Assume that the (lower and upper) value functions are bounded. Then, the ZSDG has a value, i.e., $L(t,\nu) = \mathbb{L}(t,y) = \mathbb{M}(t,y) = M(t,\nu)$ for $(t,\nu) \in [0,T] \times (\mathcal{P}_2(\mathbb{R}^{n}))^2$ and $(t,y) \in [0,T] \times (\mathcal{L}_2(\Omega, \mathbb{R}^{n}))^2$. Moreover, the value function is the unique viscosity solution to the HJI equation with $H := H^{-} = H^{+}$ in (\ref{eq_32}) and (\ref{eq_33}), and $\mathbb{H} := \mathbb{H}^{-} = \mathbb{H}^{+}$ in (\ref{eq_36}) and (\ref{eq_37}).
\end{corollary}

\begin{Remark}
Corollary \ref{Corollary_2} implies that the viscosity solution to the HJI equation characterizes the value of the ZSDG formulated in Section \ref{Section_2_2}.
\end{Remark}

\section{Numerical Examples}\label{Section_5}

This section provides two numerical examples. For the HJI equations in Section \ref{Section_4}, assume that $T=1$, $f(t,x,\nu_x,u, v) = \frac{1}{1+x^2} + \int_{\mathbb{R}} \sin(x) \dd \nu_x (x) + u - 0.1v$, $l(t,x,\nu_x, u, v) = \sin (x) + \int_{\mathbb{R}} x \dd \nu_x (x) + u - v$ and $m(x,z) = \sin (x) - z$. Also, $U = [0,1]$, $V = [0,1] $, and $x$ and $z$ are independent Gaussian random variables with mean zero and variance one (equivalently, $\nu_x$ and $\nu_z$ are Gaussian measures). Note that the target variable $z$ is included in the terminal cost $m$. The ZSDG formulated in this section can be regarded as a class of mean-field type control (Example \ref{Example_1}) and pursuit-evasion games (Example \ref{Example_3}). Due to the random initial and target variables, and the dependence of $f$ and $l$ on $\nu_x$, the problem cannot be solved using the existing theory for ZSDGs.

Note that (H.1)-(H.3) hold. Since the corresponding Hamiltonian is separable in $u$ and $v$, the Isaacs condition in (\ref{eq_50}) holds. Hence, from Corollary \ref{Corollary_2}, the ZSDG has a value that can be characterized by solving the following HJI partial differential equation (PDE) in (\ref{eq_36}) and (\ref{eq_37})\footnote{It is the lifted HJI equations in (\ref{eq_36}) and (\ref{eq_37}) when $\mathbb{H} := \mathbb{H}^{-} = \mathbb{H}^{+}$. From the definitions of  $\mathbb{H}^{-}$ and $\mathbb{H}^{+}$, the PDE is obtained after carrying out the maximization with respect to $v$ and the minimization with respect to $u$.}:
\begin{align*}	
\begin{cases}
\partial_t \mathbb{G}(t,x,z) + \mathbb{E} \bigl [ D_x \mathbb{G}(t,x,z) \frac{1}{1+x^2}  + D_x \mathbb{G}(t,x,z) \int_{\mathbb{R}} \sin(x) \dd \nu_x (x) \bigr ] = 0 \\
\mathbb{G}(1,x,z) = \mathbb{E}[\sin(x) - z] = 0
\end{cases}.
\end{align*}
Moreover, from (\ref{eq_32}) and (\ref{eq_33}), the HJI PDE above is equivalent to\footnote{It is the HJI equations in (\ref{eq_32}) and (\ref{eq_33}) when $H := H^{-} = H^{+}$.}
\begin{align*}
\begin{cases}
\partial_t G(t,\nu_x,\nu_z) + \int_{\mathbb{R}} \partial_{\nu_x} G(t,\nu_x,\nu_z)(x) \frac{1}{1+x^2} \zeta(x)  \dd x \\
\qquad \qquad + \int_{\mathbb{R}} \partial_{\nu_x} G(t,\nu_x,\nu_z)(x) \int_{\mathbb{R}} \sin(x) \dd \nu_x(x) \zeta(x)  \dd x = 0 \\
G(1,\nu_x,\nu_z) = \int_{\mathbb{R}} \sin(x) \zeta(x) \dd x - \int_{\mathbb{R}} z \zeta(z) \dd z =0
\end{cases},
\end{align*}
where $\zeta$ is the Gaussian probability density function. Here, we have utilized the fact that for any mean zero and variance one Gaussian random variable $x$ with the law $\nu_x = \mathbb{P}_{x}$, $\mathbb{E}[\sin(x)] = \int_{\mathbb{R}} \sin(x) \dd \nu_x(x) = \int_{\mathbb{R}} \sin(x) \zeta(x) \dd x = 0$ and $\mathbb{E}[x] = \int_{\mathbb{R}} x \dd \nu_x(x) = \int_{\mathbb{R}} x \zeta(x) \dd x = 0$. We can easily see that $\mathbb{G}(t,x,z) = G(t,\nu_x,\nu_z) = 0$ is the unique solution to the above PDE, which is the value of the ZSDG. This shows that the value of the ZSDG of this example is zero, which is determined by the laws (distributions) of random initial and target variables. Note that, in this example, the game value is independent of explicit values of initial and target variables.

\begin{figure}
\centering
\includegraphics[scale=0.47]{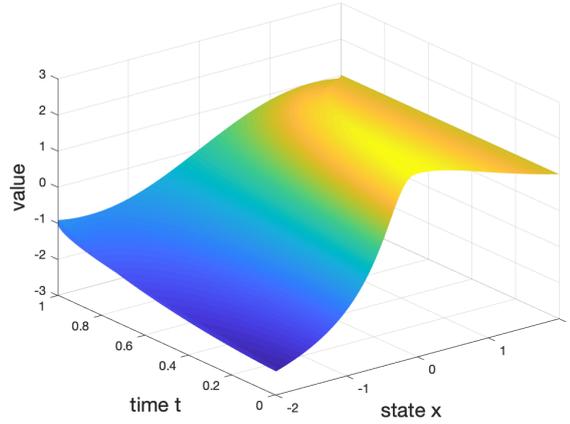}
\caption{The value of the ZSDG for the second example when $z=0$.}
\label{Fig_1}
\end{figure}

For the second example, with the same $f$, $l$ and $m$ as in the first example, assume now that $\nu_x$ and $\nu_z$ are Dirac measures. Then the associated ZSDG is reduced to the classical deterministic ZSDGs studied in \cite{Bardi_book_1997, Evans_1984, Basar2}, where the state argument of the value function is in $\mathbb{R}^n$. The HJI equation then becomes
\begin{align*}	
\begin{cases}
\partial_t \mathbb{G}(t,x,z) + D_x \mathbb{G}(t,x,z)(\frac{1}{1+x^2} + \sin(x))   + \sin(x) + x  = 0 \\
\mathbb{G}(1,x,z) = \sin(x) - z
\end{cases}.
\end{align*}
Its solution is depicted in Fig. \ref{Fig_1} when $z=0$, which is defined on $[0,1] \times [-2,2]$. In this example, we have used the finite-difference method to approximate the viscosity solution \cite{Crandall_Lions_MC_1984}. 

As seen from the two examples above and the results in the previous sections, the value of the class of ZSDGs considered in this paper depends on the laws (distributions) of random initial and target variables, whereas the values of the classical deterministic ZSDGs are determined by \emph{explicit values} of initial and target variables.

\section{Concluding Remarks}\label{Section_6}

We have studied, in this paper, a class of two-player zero-sum differential games, where the dynamical system depends on the random initial condition and the distribution of the state process, and the objective functional includes the latter as well as a random target variable. The (lower and upper) value functions are defined on two infinite-dimensional spaces, $\mathcal{P}_2$ and $\mathcal{L}_2$, which satisfy the dynamic programming principles. By using the notion of derivative in $\mathcal{P}_2$ and its lifted version in $\mathcal{L}_2$, the (lower and upper) value functions are shown to be unique (continuous) viscosity solutions to associated (lower and upper) HJI equations that are first-order PDEs on infinite-dimensional spaces. Under the Isaacs condition, the lower and upper value functions are identical, which implies that the ZSDG  has a value.

One possible future research topic would be to study the stochastic framework of ZSDGs in this paper as an extension of \cite{Cosso_JMPA_2018}, in which there is an additive Brownian noise in (\ref{eq_1}) and the corresponding diffusion term depends on the state, the law of the state process, and the control variables. This requires the notion of the second-order derivative in $\mathcal{P}_2$ and its lifted version in $\mathcal{L}_2$ to obtain DPPs, second-order HJI equations, and their viscosity solutions. Another direction would be the risk-sensitive ZSDGs. The problem of characterization of reachable sets, which can be viewed as an application of ZSDGs in this paper (see Example \ref{Example_3}), would also be an interesting avenue to pursue. In this case, the major challenge would be to solve the HJI equation numerically in the infinite-dimensional space. Finally, the extension of the rational expectations models considered in \cite{Basar_Rational_1989} to the continuous-time framework is an interesting problem to study.

\begin{appendix}
\numberwithin{equation}{section}
\renewcommand{\theequation}{\thesection.\arabic{equation}}

\section{Proof of Lemmas \ref{Lemma_1} and \ref{Lemma_2}} \label{Appendix_A}

%We first state the following useful remark:
\begin{Remark_1}\label{Remark_A_1}
\begin{enumerate}[(i)]
	\item For (\ref{eq_1}), let $x(t) = x$ be the initial condition of (\ref{eq_1}) at the initial time $t \in [0,T)$. Assume that $x$ is  distributed according to $\nu_x \in \mathcal{P}_2(\mathbb{R}^{n})$. Then the law of the state process is denoted by $\mathbb{P}_{s}^{t,\nu_x;u,v}$ for $s \in [t,T]$. Then we can easily show that $\mathbb{P}_{s}^{t,\nu_x;u,v} \in \mathcal{P}_2(\mathbb{R}^{n})$ satisfies
\begin{align}
\label{eq_5}
\mathbb{P}_{s}^{t,\nu_x;u,v} = \mathbb{P}_{s}^{r,\mathbb{P}_{r}^{t,\nu_x;u,v};u,v},~ 0 \leq t \leq r \leq s \leq T,
\end{align}
for any $\nu_x \in \mathcal{P}_2(\mathbb{R}^{n}) $, $u \in \mathcal{U}$ and $v \in \mathcal{V}$. That is, the law of the state process (\ref{eq_5}) satisfies the semigroup or flow property. 
\item We use the notation $x_{s}^{t,x,\nu_x;u,v} = x(s) $, $s \in [t,T]$, with $x_t^{t,x,\nu_x;u,v} = x(t) = x$ to emphasize the initial condition and the initial time.
\end{enumerate}
\end{Remark_1}

\begin{proof}[Proof of Lemma \ref{Lemma_1}]
		We prove (i) only, since the proof of (ii) is similar to that of (i). Consider the two initial pairs of random vectors $y = (x,z),~ \bar{y} = (\bar{x},\bar{z}) \in (\mathcal{L}_2(\Omega,\mathbb{R}^{n}))^2$ having the same law (distribution), i.e., $\nu = \mathbb{P}_{(x,z)} = \mathbb{P}_{(\bar{x},\bar{z})} \in (\mathcal{P}_2(\mathbb{R}^{n}))^2$. Since the objective functional in (\ref{eq_2}) does not depend on the random variables, but depends on the law of the initial random pair, we have $J(t,y; u,v) = J(t,\bar{y};u,v) = J(t,\nu;u,v)$ for $u \in \mathcal{U}[t,T]$ and $v \in \mathcal{V}[t,T]$. This, together with the fact that $\alpha \in \mathcal{A}[t,T]$ and $v \in \mathcal{V}[t,T]$ are not dependent on the law of the initial random pair, implies that $J(t,y; \alpha(v),v) = J(t,\bar{y};\alpha(v),v) = J(t,\nu;\alpha(v),v)$ for $\alpha \in \mathcal{A}[t,T]$ and $v \in \mathcal{V}[t,T]$. Then, from the definitions in (\ref{eq_3}) and (\ref{eq_6}), we have the desired result.
\end{proof}

%In the proof of Lemma \ref{Lemma_2}, we need the following lemma, whose proof can be found in []
%
%\begin{lemma_1}\label{Lemma_A_1}
%	Assume that $g:\mathbb{R}^n \rightarrow \mathbb{R}$ is Lipschitz continuous with Lipschitz constant $K$. Then for any $\mu_1,\mu_2 \in \mathcal{P}_2(\mathbb{R}^n)$,
%	\begin{align*}
%		\Bigl | \int_{\mathbb{R}^n} g(x) \dd \mu_1(x) - \int_{\mathbb{R}^n} g(x) \dd \mu_2(x) \Bigr | \leq K W_2(\mu_1,\nu_2)
%	\end{align*}
%\begin{Proof}
%Let $\mu \in \mathcal{P}_2(\mathbb{R}^n \times \mathbb{R}^n)$ be the optimal joint measure for $W_2(\mu_1,\mu_2)$, i.e., $\mu(A \times \mathbb{R}^n) = \mu_1(A)$ and $\mu(\mathbb{R}^n \times A) = \mu_2(A)$ for any Borel sets $A \subset \mathbb{R}^n$, which is the optimal solution of $W_2(\mu_1,\mu_2)$ (see Section \ref{Section_2_1}). Then
%\begin{align*}
%	\int_{\mathbb{R}^n} g(x) \dd \mu_1(x) & = \int_{\mathbb{R}^n \times \mathbb{R}^n}g(x) \dd \mu(x,y) \\
%	& \leq \int_{\mathbb{R}^n} g(x) \dd \mu_2(x) \\
%	&~~~ + K \Bigl ( \int_{\mathbb{R}^n \times \mathbb{R}^n} |x-y|^2 \dd \mu(x,y) \Bigr)^{1/2},
%\end{align*}
%which, together with the definition of the Wasserstein metric, leads to the desired result.
%\end{Proof}
%\end{lemma_1}

%The proof of Lemma \ref{Lemma_2} is given as follows.
\begin{proof}[Proof of Lemma \ref{Lemma_2}]
We prove here the continuity of only the lower value functions ($\mathbb{L}$ and $L$), since the proof for the upper value functions ($\mathbb{M}$ and $M$) follows along similar lines. In the proof below, a constant $c \geq 0$ can vary from line to line, which depends on the Lipschitz constant ((H.1) and (H.2)).

Let $y = (x,z)\in (\mathcal{L}_2(\Omega,\mathbb{R}^{n}))^2$ be the initial and target pair having the distribution (law) $\nu = (\nu_x,\nu_z) = \mathbb{P}_{(x,z)} \in (\mathcal{P}_2(\mathbb{R}^{n}))^2$. We apply a similar argument to the notation $y_i = (x_i,z_i)$ and $\nu_i = (\nu_{x_i},\nu_{z_i})$ for $i=1,2$. Let $t \in [0,T]$, with $t_1,t_2 \in [t,T]$. Then by using (H.1), Gronwall's lemma, and the fact that $\mathcal{P}_2(\mathbb{R}^{n}) \subset \mathcal{P}_1(\mathbb{R}^{n})$ and $W_1(\nu_{x_1},\nu_{x_2}) \leq W_2(\nu_{x_1},\nu_{x_2})$ \cite[Chapter 6]{Villani_book}, we have
\begin{align}
\label{eq_9}
\begin{cases}
\mathbb{E} \bigl [\sup_{s \in [t,T]} |x_{s}^{t,x,\nu_x;u,v}| \bigr ] \leq c (1 + \mathbb{E}[|x|]) \\
%\label{eq_10}
\mathbb{E} \bigl [ \sup_{s \in [t,T]} |x_{s}^{t,x_1,\nu_{x_1};u,v} - x_{s}^{t,x_2,\nu_{x_2};u,v} | \bigr ] \leq c \mathbb{E}[|x_1 - x_2|],  \\ 
% \label{eq_12}
\mathbb{E} \bigl [|x_{t_1}^{t,x,\nu_x;u,v} - x_{t_2}^{t,x,\nu_x;u,v} | \bigr ] \leq  c |t_1 - t_2|,  \\
% \label{eq_11}
\mathbb{E} \bigl [ \sup_{s \in [t,T]} |x_{s}^{t,x_1,\nu_{x_1};u,v} - x_{s}^{t,x_2,\nu_{x_2};u,v} | \bigr ]  \leq c W_2(\nu_{x_1},\nu_{x_2})	
\end{cases}.
\end{align}

In view of the definition of the Wasserstein metric, H\"older inequality, and the definition of the norm $\|\cdot\|_{\mathcal{L}_2}$, the preceding estimates in (\ref{eq_9}) imply that
\begin{align}
\label{eq_13_1}
& W_2(\mathbb{P}_s^{t,\nu_x;u,v}, \nu_x) \\
&=  \inf \bigl \{ \mathbb{E} \bigl [ |x_{s}^{t,x,\nu_{x};u,v} - x_{t}^{t,x,\nu_x;u,v} |^2 \bigr ]^{1/2} ~| \nonumber \\
&\qquad \qquad x_1,x_2 \in \mathcal{L}_2(\Omega,\mathbb{R}^{n})~  \text{with}~ \mathbb{P}_{x} = \nu_{x}~\text{and}~ \mathbb{P}_{x} = \nu_{x} \bigr \}  \leq c |t-s|, \nonumber
\end{align}
where the inequality follows from (\ref{eq_9}). Moreover, we have
\begin{align}
\label{eq_13}
& W_2(\mathbb{P}_s^{t,\nu_{x_1};u,v}, \mathbb{P}_s^{t,\nu_{x_2};u,v}) \\
%&=  \inf \Bigl \{ \mathbb{E} \Bigl [ |x_{s}^{t,x_1,\nu_{x_1};u,v} - x_{s}^{t,x_2,\nu_{x_2};u,v} |^2 \Bigr ]^{1/2} ~| \nonumber
%%\\
%&\qquad \qquad x_1,x_2 \in \mathcal{L}_2(\Omega,\mathbb{R}^{n})~  \text{with}~ \mathbb{P}_{x_1} = \nu_{x_1}~\text{and}~ \mathbb{P}_{x_2} = \nu_{x_2} \Bigr \} \nonumber \\
& \leq c \inf \bigl \{ \mathbb{E}[|x_1-x_2|^2]^{1/2}~|~ \nonumber \\
&\qquad \qquad x_1, x_2 \in \mathcal{L}_2(\Omega,\mathbb{R}^{n})~  \text{with}~ \mathbb{P}_{x_1} = \nu_{x_1}~\text{and}~ \mathbb{P}_{x_2} = \nu_{x_2} \bigr \} \nonumber \\
& = c W_2 (\nu_{x_1},\nu_{x_2}), \nonumber
\end{align}
where the inequality follows from (\ref{eq_9}), the definition of the $2$-Wasserstein metric, and H\"older's inequality. Then using (\ref{eq_5}), (\ref{eq_13}) and (\ref{eq_13_1}), together with the distance property of $W_2$, we have
\begin{align}
\label{eq_14_1}
& W_2(\mathbb{P}_s^{t_1,\nu_{x_1};u,v}, \mathbb{P}_s^{t_2,\nu_{x_2};u,v}) = W_2(\mathbb{P}_s^{t_1,\nu_{x_1};u,v}, \mathbb{P}_s^{t_1,\mathbb{P}_{t_1}^{t_2,\nu_{x_2};u,v};u,v}) \\
& \leq c W_2(\nu_{x_1},\mathbb{P}_{t_1}^{t_2,\nu_{x_2};u,v}) \nonumber \\
& \leq c W_2(\nu_{x_1},\nu_{x_2}) + c W_2(\nu_{x_2},\mathbb{P}_{t_1}^{t_2,\nu_{x_2};u,v}) \leq c W_2(\nu_{x_1},\nu_{x_2}) + c |t_1 - t_2|. \nonumber 
\end{align}
%where the equality follows from (\ref{eq_5}), the first inequality is due to (\ref{eq_13}), the second inequality follows from the distance property of the Wasserstein metric, and the last inequality is due to (\ref{eq_13_1}). 
Furthermore, with the estimates in (\ref{eq_9}), and (H.1) and (H.2),
%we can show that there exists a constant $c \geq 0$, dependent on the constants in (\ref{eq_9})-(\ref{eq_11})(which can vary from line to line), such that 
for any $t \in [0,T]$, $t_1,t_2 \in [t,T]$ and $y,y_1,y_2 \in (\mathcal{L}_2 (\Omega,\mathbb{R}^{n}))^2$, we have for $u \in \mathcal{U}[t,T]$ and $v \in \mathcal{V}[t,T]$,
\begin{align}
\label{eq_14}
 & |J(t_1,y;u,v) - J(t_2,y;u,v)| \leq c |t_1 - t_2|   \\
\label{eq_15}
& |J(t,y_1;u,v) - J(t,y_2;u,v)|  \\
&\leq c  \|x_1 - x_2\|_{\mathcal{L}_2} + c \|z_1 - z_2 \|_{\mathcal{L}_2} + c W_2(\nu_{x_1},\nu_{x_2}). \nonumber 
% &\leq c  \mathbb{E}[|x_1 - x_2| + |z_1 - z_2|] + c W_2(\nu_{x_1},\nu_{x_2}), \nonumber
\end{align}

The convergence in $\mathcal{L}_2(\Omega,\mathbb{R}^n)$ implies convergence in $\mathcal{P}_2(\mathbb{R}^n)$ with respect to $W_2$, i.e., $\|x_n - x\|_{\mathcal{L}_2} \rightarrow 0$  as $n \rightarrow \infty$ implies $W_2(\nu_n,\nu) \rightarrow 0$ as $n \rightarrow \infty$ \cite[Theorem 6.9]{Villani_book}. Also, $|\sup h(x) - \sup  g(x)| \leq \sup |h(x) - g(x)|$ and $|\inf h(x) - \inf  g(x)| \leq \sup |h(x) - g(x)|$. Then, from (\ref{eq_14}) and (\ref{eq_15}), and the definition of $\mathbb{L}$, we can easily see that $\mathbb{L}$ is continuous in $(t,y) \in [0,T] \times (\mathcal{L}_2(\Omega,\mathbb{R}^{n}))^2$.

For the continuity of $L$, we consider the following equivalent notation of the objective functional in terms of $\nu \in (\mathcal{P}_2(\mathbb{R}^{n}))^2$:
\begin{align*}
& J(t,\nu;u,v) 
 = \int_t^T \int_{\mathbb{R}^{n}} l(s,x_{s}^{t,x,\nu_x;u,v},\mathbb{P}_{s}^{t,\nu_x;u,v},  u(s),v(s)) \dd \mathbb{P}_{s}^{t,\nu_x;u,v} (x) \dd s \nonumber \\
&\qquad\qquad\qquad + \int_{\mathbb{R}^{n} \times \mathbb{R}^n} m(x_{T}^{t,x,\nu_x;u,v}, z) \dd   ( \mathbb{P}_{T}^{t,\nu_x;u,v}, \nu_z) (x,z). \nonumber
\end{align*}
Then, with (H.1), (H.2), (\ref{eq_9}) and (\ref{eq_14_1}), we apply the definition of the Wasserstein metric and \cite[Lemma 3]{Cardaliaguet_IGTR_2008} to show that
\begin{align}	
\label{eq_18}
|J(t,\nu_1;u,v) - J(t,\nu_2;u,v)|  \leq c (W_2(\nu_{x_1},\nu_{x_2}) + W_2(\nu_{z_1},\nu_{z_2})).  
\end{align}
Note that (\ref{eq_14}) and (\ref{eq_18}) imply that $J$ is continuous in $(t,\nu) \in [0,T] \times (\mathcal{P}_{2}(\mathbb{R}^{n}))^2$ for any $u \in \mathcal{U}[t,T]$ and $v \in \mathcal{V}[t,T]$. Then, by following the proof for the continuity of $\mathbb{L}$, we can show that $L$ is continuous in $(t,\nu) \in [0,T] \times (\mathcal{P}_{2}(\mathbb{R}^{n}))^2$. This completes the proof.
\end{proof}

\section{Proof of Theorem \ref{Theorem_1}} \label{Appendix_B}
%This appendix provides a proof for Theorem \ref{Theorem_1}.

We prove (\ref{eq_22}) only, since the proofs for (\ref{eq_20}), (\ref{eq_21}) and (\ref{eq_23}) are similar to that for (\ref{eq_22}). Let $t,t+\tau \in [0,T]$ with $t < t + \tau$. For any $\nu \in (\mathcal{P}_2(\mathbb{R}^{n}))^2$, let
\begin{align*}
& \bar{L}(t,\nu_x,\nu_z)  := \inf_{\alpha \in \mathcal{A}[t,t+\tau]} \sup_{v \in \mathcal{V}[t,t+\tau]} \Bigl \{  \int_t^{t+\tau} \int_{\mathbb{R}^{n}} l(s,x_{s}^{t,x,\nu_x;\alpha(v),v}, \mathbb{P}_{s}^{t,\nu_x;\alpha(v),v},  \\
&\qquad \qquad \alpha(v)(s),v(s)) \dd \mathbb{P}_{s}^{t,\nu_x;\alpha(v),v} (x) \dd s  + L(t+\tau,\mathbb{P}_{t+\tau}^{t,\nu_x;\alpha(v),v},\nu_z) \Bigr \}.	
\end{align*}
We need to show that $L(t,\nu_x,\nu_z) = \bar{L}(t,\nu_x,\nu_z)$. 

For any $\epsilon > 0$, there exists $\alpha^\prime \in \mathcal{A}[t,t+\tau]$ such that
\begin{align}
\label{eq_24}
& \bar{L}(t,\nu_x,\nu_z)   \\
&\geq 	\sup_{v \in \mathcal{V}[t,t+\tau]} \Bigl \{ \int_t^{t+\tau} \int_{\mathbb{R}^{n}} l(s,x_{s}^{t,x,\nu_x;\alpha^\prime(v),v}, \mathbb{P}_{s}^{t,\nu_x;\alpha^\prime(v),v}, \nonumber \\
&\qquad  \alpha^\prime(v)(s),v(s)) \dd \mathbb{P}_{s}^{t,\nu_x;\alpha^\prime(v),v} (x) \dd s  + L(t+\tau,\mathbb{P}_{t+\tau}^{t,\nu_x;\alpha^\prime(v),v},\nu_z) \Bigr \} - \epsilon. \nonumber
\end{align}
Similarly, in view of the definition of the value function, for any $\epsilon > 0$, there exists $\alpha^{\prime \prime} \in \mathcal{A}[t+\tau,T]$ such that for any $x \in \mathbb{R}^n$ with $\nu_x \in \mathcal{P}_2(\mathbb{R}^{n})$, 
\begin{align}
\label{eq_25}	
& L(t+\tau,\mathbb{P}_{t+\tau}^{t,\nu_x;\alpha^\prime(v),v},\nu_z) + \epsilon \\
&\geq 	\sup_{v \in \mathcal{V}[t+\tau,T]} \Bigl \{ \int_{t+\tau}^{T} \int_{\mathbb{R}^{n}} l(s,x_{s}^{t+\tau,x,\nu_x;\alpha^{\prime \prime}(v),v}, \mathbb{P}_{s}^{t+\tau,\nu_x;\alpha^{\prime \prime}(v),v},\nonumber \\
& \qquad\qquad\qquad \alpha^{\prime \prime}(v)(s),v(s)) \dd \mathbb{P}_{s}^{t+\tau,\nu_x;\alpha^{\prime \prime}(v),v} (x) \dd s  \nonumber \\
& \qquad +  \int_{\mathbb{R}^{n} \times \mathbb{R}^n} m(x_T^{t+\tau,x,\nu_x;\alpha^{\prime \prime}(v),v} , z) \dd   (\mathbb{P}_T^{t+\tau,\nu_x;\alpha^{\prime \prime}(v),v},\nu_z)(x,z) \Bigr \}  \nonumber
\end{align}
Define $\alpha \in \mathcal{A}[t,T]$ such that for $v \in \mathcal{V}[t,T]$, $\alpha(v)(s) = \alpha^\prime(v)(s)$ on $s \in [t,t+\tau)$ and $\alpha(v)(s) = \alpha^{\prime \prime}(v)(s)$ on $s \in [t+\tau,T]$. Then, from (\ref{eq_5}), (\ref{eq_24}) and (\ref{eq_25}), we can show that
\begin{align*}
& \bar{L}(t,\nu_x,\nu_z) + 2 \epsilon	 \\
& \geq \int_{t}^{T} \int_{\mathbb{R}^{n}} l(s,x_{s}^{t,x,\nu_x;\alpha(v),v}, \mathbb{P}_{s}^{t,\nu_x;\alpha(v),v}, \alpha(v)(s),v(s)) \dd \mathbb{P}_{s}^{t,\nu_x;\alpha(v),v} (x) \dd s   \\
&\qquad +  \int_{\mathbb{R}^{n} \times \mathbb{R}^n} m(x_T^{t,x,\nu_x;\alpha(v),v}, z)\dd (\mathbb{P}_T^{t,\nu_x;\alpha(v),v}, \nu_z)(x,z), 
\end{align*}
which implies
\begin{align}
\label{eq_26}
L(t,\nu) \leq \bar{L}(t,\nu) + 2\epsilon.
\end{align}

On the other hand, for any $\epsilon > 0$ and $v \in \mathcal{V}[t,T]$, there exists $\alpha^\prime \in \mathcal{A}[t,T]$ such that
\begin{align}
\label{eq_27}
& L(t,\nu_x,\nu_z) + \epsilon  \geq \sup_{v \in \mathcal{V}[t,T]}	\Bigl \{   \int_{t}^{T} \int_{\mathbb{R}^{n}} l(s,x_{s}^{t,x,\nu_x;\alpha^\prime(v),v}, \mathbb{P}_{s}^{t,\nu_x;\alpha^\prime(v),v}, \alpha^\prime(v)(s),\\
&~~~~~~~~~~~~~v(s)) \dd \mathbb{P}_{s}^{t,\nu_x;\alpha^\prime(v),v} (z) \dd s \Bigr \} + \int_{\mathbb{R}^{n} \times \mathbb{R}^n} m(x_T^{t,x,\nu_x;\alpha^\prime(v),v} , z) \dd (\mathbb{P}_T^{t,\nu_x;\alpha^\prime(v),v}, \nu_z)(x,z),  \nonumber
\end{align}
and by restricting $\alpha^\prime$ to $[t,t+\tau]$, we have
\begin{align}
\label{eq_28}
& \bar{L}(t,\nu_x,\nu_z) \\
& \leq 	\sup_{v \in \mathcal{V}[t,t+\tau]} \Bigl \{  L(t+\tau,\mathbb{P}_{t+\tau}^{t,\nu_x;\alpha(v),v},\nu_z) + \int_t^{t+\tau} \int_{\mathbb{R}^{n}} l(s,x_{s}^{t,x,\nu_x;\alpha^\prime(v),v}, \nonumber \\
&\qquad \qquad \mathbb{P}_{s}^{t,\nu_x;\alpha^\prime(v),v}, \alpha^\prime(v)(s),v(s)) \dd \mathbb{P}_{s}^{t,\nu_x;\alpha^\prime(v),v} (x) \dd s \Bigr \}. \nonumber
\end{align} 
The inequality in (\ref{eq_28}) implies that for each $\epsilon > 0$, there exists $v^\prime \in \mathcal{V}[t,t+\tau]$ such that
\begin{align}
\label{eq_29}	
& \bar{L}(t,\nu_x,\nu_z) \leq L(t+\tau,\mathbb{P}_{t+\tau}^{t,\nu_x;\alpha^\prime(v^\prime),v^\prime},\nu_z) +\int_t^{t+\tau} \int_{\mathbb{R}^{n}} l(s,x_{s}^{t,x,\nu_x;\alpha^\prime(v^\prime),v^\prime}, \\
& \qquad \qquad \mathbb{P}_{s}^{t,\nu_x;\alpha^\prime(v^\prime),v^\prime}, \alpha^\prime(v^\prime)(s), v^\prime(s)) \dd \mathbb{P}_{s}^{t,\nu_x;\alpha^\prime(v^\prime),v^\prime} (x) \dd s + \epsilon.  \nonumber
\end{align}
Similarly, for any $\epsilon > 0$, there exists $v^{\prime \prime} \in \mathcal{V}[t+\tau,T]$ such that for any $x \in \mathbb{R}^n$ with $\nu_x \in \mathcal{P}_2(\mathbb{R}^{n})$
\begin{align}
\label{eq_30}
& L(t+\tau,\mathbb{P}_{t+\tau}^{t,\nu_x;\alpha^\prime(v^\prime),v^\prime},\nu_z) \\
& \leq   \int_{\mathbb{R}^{n} \times \mathbb{R}^n} m(x_T^{t+\tau,x,\nu_x;\alpha^\prime(v^{\prime \prime}),v^{\prime \prime}},z) \dd (\mathbb{P}_T^{t+\tau,\nu_x;\alpha^\prime(v^{\prime \prime}),v^{\prime \prime}},\nu_z)(x,z)  \nonumber   \\
&\qquad+ \int_{t+\tau}^{T}  \int_{\mathbb{R}^{n}} l(s,x_{s}^{t+\tau,x,\nu_x;\alpha^\prime(v^{\prime \prime}),v^{\prime \prime}},\mathbb{P}_{s}^{t+\tau,\nu_x;\alpha^\prime(v^{\prime \prime}),v^{\prime \prime}}, \nonumber \\
&\qquad \qquad \qquad \alpha^\prime(v^{\prime \prime})(s),v^{\prime \prime}(s)) \dd \mathbb{P}_{s}^{t+\tau,\nu_x;\alpha^\prime(v^{\prime \prime}),v^{\prime \prime}} (x) \dd s  \ + \epsilon. \nonumber
\end{align}
We define $\bar{v} \in \mathcal{V}[t,T]$ such that $\bar{v}(s) = v^\prime(s)$ on $s \in [t,t+\tau)$ and $\bar{v}(s) = v^{\prime \prime}(s)$ on $s \in [t+\tau, T]$. Consider $\alpha^\prime \in \mathcal{A}[t,T]$ with $\bar{v}$, i.e., $\alpha^\prime(\bar{v})(s) = \alpha^\prime(v^\prime)(s)$ on $s \in [t, t+\tau)$ and $\alpha^\prime(\bar{v})(s) = \alpha^\prime(v^{\prime \prime})(s)$ on $s \in [t+\tau, T]$. Then, from (\ref{eq_5}), (\ref{eq_29}) and (\ref{eq_30}), 
\begin{align*}
 \bar{L}(t,\nu_x,\nu_z) & \leq \int_{\mathbb{R}^{n} \times \mathbb{R}^n} m(x_T^{t,x,\nu_x;\alpha^\prime(\bar{v}),\bar{v}},z) \dd (\mathbb{P}_T^{t,\nu_x;\alpha^\prime(\bar{v}),\bar{v}}, \nu_z)(x,z) \\
&\qquad + \int_{t}^{T}  \int_{\mathbb{R}^{n}} l(s,x_{s}^{t,x,\nu_x;\alpha^\prime(\bar{v}),\bar{v}}, \mathbb{P}_{s}^{t,\nu_x;\alpha^\prime(\bar{v}),\bar{v}},\\
&\qquad \qquad \qquad \alpha^\prime(\bar{v})(s),\bar{v}(s)) \dd \mathbb{P}_{s}^{t,\nu_x;\alpha^\prime(\bar{v}),\bar{v}} (x) \dd s  + 2\epsilon.   
\end{align*}
This, together with (\ref{eq_27}), implies that
\begin{align}
\label{eq_31}
\bar{L}(t,\nu) \leq L(t,\nu) + 3 \epsilon.
\end{align}
Since $\epsilon$ was arbitrary, in view of (\ref{eq_26}) and (\ref{eq_31}), $L(t,\nu) = \bar{L}(t,\nu)$; hence, we have the desired result.

\section{Proof of Theorem \ref{Theorem_2}} \label{Appendix_C}

%This appendix provides a proof for Theorem \ref{Theorem_2}.

In view of Lemma \ref{Lemma_2}, $L \in \mathcal{C}([0,T] \times \mathcal{P}_2(\mathbb{R}^n))$. We now prove that the value function $L$ in (\ref{eq_3}) is a viscosity supersolution of (\ref{eq_32}). From the definition of the value function and (\ref{eq_22}) in Theorem \ref{Theorem_1}, we have $\phi(T,\nu) = L(T,\nu)$.
		
From the definition of the viscosity supersolution (Definition \ref{Definition_2}(i)), for any $\phi \in \mathcal{C}^{1,1}([0,T] \times \mathcal{P}_2(\mathbb{R}^n))$, $L(t,\nu) - \phi(t,\nu) \leq L(t^\prime,\nu^\prime) - \phi(t^\prime,\nu^\prime)$ for all $(t^\prime,\nu^\prime)$ with $|t - t^\prime| + W_2(\nu,\nu^\prime) \leq \delta$ with some $\delta > 0$. Let $t^\prime = t+\tau$ and $\nu^\prime = (\nu_x^\prime, \nu_z)$ satisfying $|\tau| + W_2(\nu_x,\nu_x^\prime) \leq \delta$. Moreover, without loss of generality, we may assume $L(t,\nu) = \phi(t,\nu)$. 

Then in view of the DPP of (\ref{eq_22}) in Theorem \ref{Theorem_1},
	\begin{align*}
		\phi(t,\nu_x,\nu_z) &= L(t,\nu_x,\nu_z) \\
		& = \inf_{\alpha \in \mathcal{A}[t,t+\tau]} \sup_{v \in \mathcal{V}[t,t+\tau]} \Bigl \{  \int_t^{t+\tau} \int_{\mathbb{R}^{n}} l(s,x_{s}^{t,x,\nu_x;\alpha(v),v}, \mathbb{P}_{s}^{t,\nu_x; \alpha(v),v}, \nonumber \\
		&\qquad \qquad  \alpha(v)(s),v(s)) \dd \mathbb{P}_{s}^{t,\nu_x;\alpha(v),v} (x) \dd s  + L(t+\tau,\mathbb{P}_{t+\tau}^{t,\nu_x;\alpha(v),v},\nu_z) \Bigr \}, \nonumber
	\end{align*}
	and  
	\begin{align}
	\label{eq_40}
		&  \inf_{\alpha \in \mathcal{A}[t,t+\tau]} \sup_{v \in \mathcal{V}[t,t+\tau]} \Bigl \{  \int_t^{t+\tau} \int_{\mathbb{R}^{n}} l(s,x_{s}^{t,x,\nu_x;\alpha(v),v},\mathbb{P}_{s}^{t,\nu_x; \alpha(v),v},\alpha(v)(s), \\
		& v(s)) \dd \mathbb{P}_{s}^{t,\nu_x;\alpha(v),v} (x) \dd s  + \phi(t+\tau,\mathbb{P}_{t+\tau}^{t,\nu_x;\alpha(v),v},\nu_z) \Bigr \} - \phi(t,\nu_x,\nu_z) \leq 0. \nonumber 
	\end{align}	
	For $\alpha \in \mathcal{A}[t,T]$ and $v \in \mathcal{V}[t,T]$, $\inf_{u \in \mathcal{U}[t,T]} J(t,\nu;u,v) \leq \sup_{v \in \mathcal{V}[t,T]} J(t,\nu;\alpha(v), v)$, which implies that 
	\begin{align*}
	& \sup_{v \in \mathcal{V}[t,T]} \inf_{u \in \mathcal{U}[t,T]} J(t,\nu;u,v)  \leq \inf_{\alpha \in \mathcal{A}[t,T]} \sup_{v \in \mathcal{V}[t,T]} J(t,\nu;\alpha(v),v) = L(t,\nu).
	\end{align*}
%$\sup_{v \in \mathcal{V}[t,T]} \inf_{u \in \mathcal{U}[t,T]} J(t,\nu;u,v)  \leq \inf_{\alpha \in \mathcal{A}[t,T]} \sup_{v \in \mathcal{V}[t,T]} J(t,\nu;\alpha(v),v) = L(t,\nu)$. 
Hence, with (\ref{eq_40}), we have
	\begin{align*}
	& \sup_{v \in \mathcal{V}[t,t+\tau]} \inf_{u \in \mathcal{U}[t,t+\tau]} \Bigl \{  \int_t^{t+\tau} \int_{\mathbb{R}^{n}} l(s,x_{s}^{t,x,\nu_x;u,v}, \mathbb{P}_{s}^{t,\nu_x; u,v}, u(s),\nonumber \\
		&\qquad \qquad v(s)) \dd \mathbb{P}_{s}^{t,\nu_x;u,v} (x) \dd s + \phi(t+\tau,\mathbb{P}_{t+\tau}^{t,\nu_x;u,v},\nu_z) \Bigr \} - \phi(t,\nu_x,\nu_z) \leq 0.	
	\end{align*}
	
	For each $\epsilon > 0$ and small $\tau$ with $|\tau| \leq \delta$, there exists $u^\prime \in \mathcal{U}[t,t+\tau]$ such that for $ v \in \mathcal{V}[t,t+\tau]$,
	\begin{align*}
	& \int_t^{t+\tau} \int_{\mathbb{R}^{n}} l(s,x_{s}^{t,x,\nu_x;u^\prime,v}, \mathbb{P}_{s}^{t,\nu_x; u^\prime,v}, u^\prime(s),v(s)) \dd \mathbb{P}_{s}^{t,\nu_x;u^\prime,v} (x) \dd s\nonumber \\
		&\qquad \qquad + \phi(t+\tau,\mathbb{P}_{t+\tau}^{t,\nu_x;u^\prime,v},\nu_z) - \phi(t,\nu_x,\nu_z) \leq \epsilon \tau.
	\end{align*}
	We multiply the above expression by $\frac{1}{\tau}$, and let $\tau \downarrow 0$ and $\epsilon \downarrow 0$. Then, with the chain rule in $\mathcal{P}_2$ in Section \ref{Section_2_1},
	\begin{align}
	\label{eq_41}
	&  \partial_t \phi(t,\nu_x,\nu_z)  + \sup_{v \in V} \inf_{u \in U} \Bigl \{ \int_{\mathbb{R}^n} [ \langle \partial_{\nu_x} \phi(t,\nu_x,\nu_z)(x), f(t,x,\nu_x,u,v) \rangle \\
	&\qquad \qquad \qquad + l(t,x,\nu_x,u,v) ]\dd \nu_x(x) \Bigr \} \leq 0, \nonumber
	\end{align}
	which, together with (\ref{eq_34}), shows that $L$ is a viscosity supersolution to (\ref{eq_32}). 

We now prove, by contradiction, that $L$ is a viscosity subsolution of (\ref{eq_32}). From the definition of the viscosity subsolution (Definition \ref{Definition_2}(i)), for any $\phi \in \mathcal{C}^{1,1}([0,T] \times \mathcal{P}_2(\mathbb{R}^n))$, $L(t,\nu) - \phi(t,\nu) \geq L(t^\prime,\nu^\prime) - \phi(t^\prime,\nu^\prime)$ for all $(t^\prime,\nu^\prime)$ with $|t - t^\prime| + W_2(\nu,\nu^\prime) \leq \delta$ with some $\delta > 0$. Let $t^\prime = t+\tau$, and $\nu^\prime = (\nu_x^\prime, \nu_z)$ satisfying $|\tau| + W_2(\nu_x,\nu_x^\prime) \leq \delta$. Moreover, without loss of generality, we may assume $L(t,\nu) = \phi(t,\nu)$.

Let us assume that $L$ is not a viscosity subsolution of (\ref{eq_32}). Then, there exists a constant $\theta > 0$ such that
\begin{align*}
	\partial_t \phi(t,\nu_x,\nu_z) + H^{-}(t,\nu,\partial_{\nu_x}\phi(t,\nu_x,\nu_z)) \leq  - \theta < 0.
\end{align*}
Define $\bar{H}(t,\nu,p,u,v) :=\int_{\mathbb{R}^{n}} [ \langle p,f(t,x,\nu_x,u,v) \rangle   + l(t,x,\nu_x,u,v) ] \dd \nu_x(x)$. From (\ref{eq_34}), note that $H^{-}(t,\nu,p) = \sup_{v \in V} \inf_{u \in U} \bar{H}(t,\nu,p,u,v)$. Since $f$ and $l$ are (uniformly) continuous on $[0,T] \times U \times V$, so is $\bar{H}$, which implies that there is a measurable function $\eta : V \rightarrow U$ and $\tau_0 \in [0,T-t]$ such that for $v \in V$ and $|s-t| \leq \tau_0$,
\begin{align*}
\partial_t \phi(s,\nu_x,\nu_z) + \bar{H}(s,\nu,\partial_{\nu_x}\phi(t,\nu_x,\nu_z),\eta(v),v) \leq - \theta / 2. %\frac{\theta}{2}.	
\end{align*}

On the other hand, the DPP in (\ref{eq_22}) of Theorem \ref{Theorem_1} and the definition of the viscosity subsolution imply that
\begin{align*}
		&  \inf_{\alpha \in \mathcal{A}[t,t+\tau]} \sup_{v \in \mathcal{V}[t,t+\tau]} \Bigl \{  \int_t^{t+\tau} \int_{\mathbb{R}^{n}} l(s,x_{s}^{t,x,\nu_x;\alpha(v),v}, \mathbb{P}_{s}^{t,\nu_x; \alpha(v),v}, \alpha(v)(s), \nonumber \\
		&\qquad  v(s)) \dd \mathbb{P}_{s}^{t,\nu_x;\alpha(v),v} (x) \dd s  + \phi(t+\tau,\mathbb{P}_{t+\tau}^{t,\nu_x;\alpha(v),v},\nu_z) \Bigr \} - \phi(t,\nu_x,\nu_z) \geq 0,
	\end{align*}
	and by defining $\alpha^\prime(v(s)):= \eta(v)$ for $s \in [t, t+\tau]$, we have $\alpha^\prime(v) \in \mathcal{A}[t,t+\tau]$ and
	\begin{align*}
& 	\sup_{v \in \mathcal{V}[t,t+\tau]} \Bigl \{  \int_t^{t+\tau} \int_{\mathbb{R}^{n}} l(s,x_{s}^{t,x,\nu_x;\alpha^\prime(v),v},\mathbb{P}_{s}^{t,\nu_x; \alpha^\prime(v),v}, \alpha^\prime(v)(s), \nonumber \\
		&\qquad  v(s)) \dd \mathbb{P}_{s}^{t,\nu_x;\alpha^\prime(v),v} (x) \dd s  + \phi(t+\tau,\mathbb{P}_{t+\tau}^{t,\nu_x;\alpha^\prime(v),v},\nu_z) \Bigr \} - \phi(t,\nu_x,\nu_z) \geq 0.	
	\end{align*}
	Then, for each $\epsilon > 0$, there exists $v^\prime \in \mathcal{V}[t,t+\tau]$ such that
	\begin{align*}
	& \int_t^{t+\tau} \int_{\mathbb{R}^{n}} l(s,x_{s}^{t,x,\nu_x;\alpha^\prime(v),v^\prime}, \mathbb{P}_{s}^{t,\nu_x; \alpha^\prime(v),v^\prime}, \alpha^\prime(v^\prime)(s),v^\prime(s)) \dd \mathbb{P}_{s}^{t,\nu_x;\alpha^\prime(v^\prime),v^\prime} (x) \dd s \nonumber \\
		&\qquad \qquad  + \phi(t+\tau,\mathbb{P}_{t+\tau}^{t,\nu_x;\alpha^\prime(v^\prime),v^\prime},\nu_z)  - \phi(t,\nu_x,\nu_z) \geq - \epsilon \tau.
	\end{align*}
Multiplying the above expression by $\frac{1}{\tau}$, and letting $\tau \downarrow 0$, together with the chain rule in $\mathcal{P}_2$ in Section \ref{Section_2_1}, yield
\begin{align*}
-\epsilon &\leq \partial_t \phi(t,\nu_x,\nu_z)  + \bar{H}(t,\nu,\partial_{\nu_x}\phi(t,\nu_x,\nu_z),\eta(v^\prime),v^\prime) \leq - \theta / 2, %\frac{\theta}{2},
\end{align*}
and by letting $\epsilon \downarrow 0$, we must have $\theta \leq  0$, which leads to a contradiction. This implies that
\begin{align}
\label{eq_42}
\partial_t \phi(t,\nu_x,\nu_z) + H^{-}(t,\nu,\partial_{\nu_x}\phi(t,\nu_x,\nu_z)) \geq 0.	
\end{align}
Hence, (\ref{eq_41}) and (\ref{eq_42}) taken together show that $L$ is a viscosity solution to (\ref{eq_32}). The proof of $M$ being a viscosity solution to (\ref{eq_33}) is similar. This completes the proof.

\section{Proof of Theorem \ref{Theorem_3}}\label{Appendix_D}

In the proof of Theorem \ref{Theorem_3}, we need the following lemma, which follows from (H.1)-(H.3).

\begin{lemma_1}\label{Lemma_Appendix_D_1}
Assume that (H.1)-(H.3) hold. Then, the following result holds: there is a constant $c$, dependent on the Lipschitz constant, such that for any $y = (x,z), ~y_1 = (x_1,z_1), ~y_2 = (x_2,z_2) \in (\mathcal{L}_2(\Omega,\mathbb{R}^{n}))^2$ and $p,p_1,p_2 \in \mathcal{L}_2(\Omega,\mathbb{R}^{n})$, we have $|\mathbb{H}^{-}(y_1,p) - \mathbb{H}^{-}(y_2,p)|  \leq c (1 + \|p\|_{\mathcal{L}_2}) \|x_1 - x_2\|_{\mathcal{L}_2}$ and $|\mathbb{H}^{-}(y,p_1) - \mathbb{H}^{-}(y,p_2)|  \leq c \|p_1 - p_2 \|_{\mathcal{L}_2}$.
%\begin{align*}
%|\mathbb{H}^{-}(y_1,p) - \mathbb{H}^{-}(y_2,p)| & \leq c (1 + \|p\|_{\mathcal{L}_2}) \|x_1 - x_2\|_{\mathcal{L}_2} \\
%|\mathbb{H}^{-}(y,p_1) - \mathbb{H}^{-}(y,p_2)| & \leq c \|p_1 - p_2 \|_{\mathcal{L}_2}.
%\end{align*}
%\begin{proof}
%The proof is a direct application of (H.1) and (H.2).	
%\end{proof}
%
%\begin{Proof}
%Let
%\begin{align*}
%& \mathbb{H}(t,y,p,u,v) \\
%& = \mathbb{E} \Bigl [ \langle p,f(t,x,\nu_x,u,v) \rangle  + l(t,x,\nu_x,u,v) \Bigr ]	
%\end{align*}
%Then in view of (H.1) and (H.3), we have
%\begin{align*}
%& | \mathbb{H}(t,y_1,p,u,v) - 	\mathbb{H}(t,y_2,p,u,v)	| \\
%&\leq c (1 + \|p\|_{\mathcal{L}_2}) \|x_1 - x_2\|_{\mathcal{L}_2},
%\end{align*}
%and 
%\begin{align*}
% | \mathbb{H}(t,y,p_1,u,v) - 	\mathbb{H}(t,y,p_2,u,v)	|
% \leq c \|p_1 - p_2 \|_{\mathcal{L}_2}.
%\end{align*}
%Since the above inequalities hold for any $u \in U$ and $v \in V$, we have the desired result.
%\end{Proof}
\end{lemma_1}

The proof of Theorem \ref{Theorem_3} now proceeds as follows.
\begin{proof}[Proof of Theorem \ref{Theorem_3}]
We first prove $\mathbb{L}_1(t,y) \leq \mathbb{L}_2(t,y)$ for $(t,y) \in [0,T] \times (\mathcal{L}_2(\Omega, \mathbb{R}^{n}))^2$. Note that both $\mathbb{L}_1$ and $\mathbb{L}_2$ are bounded by some constant $c$. In the proof below, a constant $c$ can vary from line to line, depending on the bounds of $\mathbb{L}_1$ and $\mathbb{L}_2$, and the Lipschitz constant in (H.1) and (H.2).

By a possible abuse of notation, we reverse the time by defining $t_i := T - t_i^\prime$, where $t_i^\prime \in [0,T]$, $i=1,2$. Then $\mathbb{L}_i(0,y) = \mathbb{E}[m(x,z)]$, $i=1,2$. With the time reverse notation and Remark \ref{Remark_4}(ii) (see \cite[Chapter 10]{Evans_Book}), for the lifted HJI equations, the inequality in Definition \ref{Definition_2} has to be modified by
	\begin{align}
	\label{eq_appendix_d_1_1_1}
	\begin{cases}
	\partial_t \phi(t,y) - \mathbb{H}^{-}(t,y,D_x \phi(t,x,z)) \leq 0 & 		\text{(subsolution)} \\
	\partial_t \phi(t,y) - \mathbb{H}^{-}(t,y,D_x \phi(t,x,z))  \geq 0 & \text{(supersolution)}
	\end{cases}.
	\end{align}

For $(\epsilon, \sigma,\alpha) \in (0,1)$, define 
\begin{align*}
\Phi(t_1,y_1,t_2,y_2)  :=&  \mathbb{L}_1(t_1, y_1) - \mathbb{L}_2(t_2, y_2) \\ 
&  - \frac{1}{2 \epsilon} ((t_1-t_2)^2 + \|y_1 - y_2 \|_{\mathcal{L}_2}^2)  - \frac{\alpha}{2} ( \|y_1 \|_{\mathcal{L}_2}^2 + \|y_2 \|_{\mathcal{L}_2}^2) - \sigma t_1,
\end{align*}
where $y_1,y_2 \in (\mathcal{L}_2(\Omega, \mathbb{R}^{n}))^2$. We can see that $\Phi$ is continuous on $X = ([0,T] \times (\mathcal{L}_2(\Omega, \mathbb{R}^{n}))^2)^2$, where $X$ is a Hilbert space and its dual space $X^*$ is $X^* = X$ \cite{Luenberger_book, Conway_2000_book}. For $\zeta_{t_1}, \zeta_{t_2} \in [0,T]$ and $\zeta_{y_1}, \zeta_{y_2} \in (\mathcal{L}_2(\Omega, \mathbb{R}^{n}))^2$, i.e., $(\zeta_{t_1},\zeta_{y_1}, \zeta_{t_2}, \zeta_{y_2}) \in X^*$, let us define the linearly perturbed map of $\Phi$:
\begin{align*}	
\Phi^\prime(t_1,y_1,t_2, y_2) := & \Phi(t_1,y_1,t_2, y_2) - \zeta_{t_1} t_1 - \zeta_{t_2} t_2 \\
&  - \mathbb{E}[\langle \zeta_{y_1},y_1 \rangle] - \mathbb{E}[\langle \zeta_{y_2},y_2 \rangle].
\end{align*}
Then in view of Stegall's theorem \cite{Stegall_1978, Fabian_NA_2007} and Riesz representation theorem \cite{Conway_2000_book}, there exist $(\zeta_{t_1},\zeta_{y_1}, \zeta_{t_2}, \zeta_{y_2}) \in X^*$ such that $|\zeta_{t_i}| \leq \delta$ , $\|\zeta_{y_i}\|_{\mathcal{L}_2} \leq \delta$, $i=1,2$, with $\delta \in (0,1)$, and $\Phi^\prime$ has a maximum at a point $(\bar{t}_1,\bar{y}_1,\bar{t}_2,\bar{y}_2) \in X$.\footnote{In fact, $-\Phi$ is continuous, coercive and $-\Phi:X \rightarrow [0, \infty]$; hence, in view of \cite{Fabian_NA_2007}, $-\Phi^\prime$ admits a minimum, i.e., $\Phi^\prime$ admits a maximum.} This implies that
\begin{align*}
\Phi^\prime(\bar{t}_1, \bar{y}_1,\bar{t}_2, \bar{y}_2) \geq \Phi^\prime(0,0,0,0) = \Phi(0,0,0,0),
\end{align*}
which, together with Cauchy-Schwarz inequality and the fact that $a+b \leq \sqrt{2} (a^2 + b^2)^{1/2}$ for $a,b \geq 0$, implies (note that $\mathbb{L}_1$ and $\mathbb{L}_2$ are bounded) 
\begin{align*}
& \alpha (\|\bar{y}_1\|_{\mathcal{L}_2}^2 + \|\bar{y}_2\|_{\mathcal{L}_2}^2) + \frac{1}{\epsilon} ((\bar{t}_1-\bar{t}_2)^2 + \|\bar{y}_1-\bar{y}_2 \|_{\mathcal{L}_2}^2) \\
& \leq c(1+\delta ( \|\bar{y}_1\|_{\mathcal{L}_2}^2 + \|\bar{y}_2\|_{\mathcal{L}_2}^2)^{1/2}).
\end{align*}
We apply the quadratic analysis to the above inequality. Then
\begin{align}
\label{eq_45}
	(\|\bar{y}_1\|_{\mathcal{L}_2}^2 + \|\bar{y}_2 \|_{\mathcal{L}_2}^2)^{1/2} & \leq c \Bigl (\frac{\delta}{\alpha} + \frac{1}{\alpha^{1/2}} \Bigr ) \\
\label{eq_46}
 ((\bar{t}_1-\bar{t}_2)^2 + \|\bar{y}_1-\bar{y}_2\|_{\mathcal{L}_2}^2)^{1/2}  & \leq c \epsilon^{1/2} + c \epsilon^{1/2} \delta^{1/2} \Bigl (\frac{\delta}{\alpha} + \frac{1}{\alpha^{1/2}} \Bigr )^{1/2}.
\end{align}

We show that either $\bar{t}_1 = 0$ or $\bar{t}_2 = 0$ by contradiction. Assume that $\bar{t}_i > 0$ for $i=1,2$. By defining
\begin{align*}
\phi_1(t_1,y_1) :=& \mathbb{L}_2(\bar{t}_1,\bar{y}_2)	 + \frac{1}{2\epsilon} ((t_1 - \bar{t}_2)^2 + \|y_1 - \bar{y}_2 \|_{\mathcal{L}_2}^2) \\
& + \frac{\alpha}{2} ( \|y_1 \|_{\mathcal{L}_2}^2 + \|\bar{y}_2 \|_{\mathcal{L}_2}^2) + \sigma t_1  \\
& + \zeta_{t_1} t_1 + \zeta_{t_2} \bar{t}_2 + \mathbb{E}[\langle \zeta_{y_1},y_1\rangle ] + \mathbb{E}[\langle \zeta_{y_2},\bar{y}_2\rangle ],
\end{align*}
we have $\Phi^\prime( t_1, y_1,\bar{t}_2,\bar{y}_2) = \mathbb{L}_1(t_1,y_1) - \phi_1 (t_1,y_1)$, which admits a maximum at $(\bar{t}_1,\bar{y}_1)$. Note that $\mathbb{L}_1$ is the viscosity subsolution and $\mathbb{H}^{-}$ is independent of $t$ (see (H.3)). This, together with (\ref{eq_appendix_d_1_1_1}), implies
\begin{align}
\label{eq_47}
& \frac{1}{\epsilon} (\bar{t}_1 - \bar{t}_2) + \sigma + \zeta_{t_1}   - \mathbb{H}^{-}(\bar{y}_1,\frac{1}{\epsilon} (\bar{y}_1 - \bar{y}_2) + \alpha \bar{y}_1 + \zeta_{y_1}) \leq 0. 
%& + \mathbb{H}^{-}(\bar{t}_1,\bar{y}_1,\frac{1}{\epsilon} (\bar{y}_1 - \bar{y}_2) + \alpha \bar{y}_1 + \zeta_{y_1}) \geq 0.
\end{align}
The inequality is reversed due to the time reverse notation. Similarly, we have $\Phi^\prime(\bar{t}_1,\bar{y}_1,t_2,y_2) = \phi_2(t_2,y_2) - \mathbb{L}_2(t_2,y_2)$, where
\begin{align*}
\phi_2(t_2,y_2) :=& \mathbb{L}_1(\bar{t}_1,\bar{y}_1)  - \frac{1}{2\epsilon}((\bar{t}_1 - t_2)^2 + \|\bar{y}_1 - y_2 \|^2) \\
&  - \frac{\alpha}{2}( \|\bar{y}_1 \|_{\mathcal{L}_2}^2 + \|y_2 \|_{\mathcal{L}_2}^2) - \sigma \bar{t}_1 \\
& - \zeta_{t_1} \bar{t}_1 - \zeta_{t_2} t_2 - \mathbb{E}[\langle \zeta_{y_1}, \bar{y}_1 \rangle ] - \mathbb{E}[\langle \zeta_{y_2}, y_2 \rangle ],
\end{align*}
%we have 
which admits a maximum at $(\bar{t}_2,\bar{y}_2)$, i.e., $\mathbb{L}_2(t_2,y_2) - \phi_2(t_2,y_2)$ admits a minimum at $(\bar{t}_2,\bar{y}_2)$. Then, from (\ref{eq_appendix_d_1_1_1}), we have
%Note that $\mathbb{L}_2$ is the viscosity supersolution. This, together with the definition of the viscosity supersolution, implies that
\begin{align}
\label{eq_48}
& \frac{1}{\epsilon}(\bar{t}_1 - \bar{t}_2) - \zeta_{t_2}  - \mathbb{H}^{-}(\bar{y}_2,\frac{1}{\epsilon}(\bar{y}_1 - \bar{y}_2) - \alpha \bar{y}_2 - \zeta_{y_2}) \geq   0.
%& + \mathbb{H}^{-}(\bar{t}_2,\bar{y}_2,\frac{1}{\epsilon}(\bar{y}_1 - \bar{y}_2) - \alpha \bar{y}_2 - \zeta_{y_2}) \leq 0.
\end{align}
%Note also that the inequality is reversed due to the time reverse notation.

(\ref{eq_47}) and (\ref{eq_48}) imply that
\begin{align*}
& \sigma + \zeta_{t_1} + \zeta_{t_2} - \mathbb{H}^{-}(\bar{y}_1,\frac{1}{\epsilon} (\bar{y}_1 - \bar{y}_2) + \alpha \bar{y}_1 + \zeta_{y_1}) \\
& + \mathbb{H}^{-}(\bar{y}_2,\frac{1}{\epsilon}(\bar{y}_1 - \bar{y}_2) - \alpha \bar{y}_2 - \zeta_{y_2}) \leq 0.
%& + \mathbb{H}^{-}(\bar{t}_1,\bar{y}_1,\frac{1}{\epsilon} (\bar{y}_1 - \bar{y}_2) + \alpha \bar{y}_1 + \zeta_{y_1}) \\
%& - \mathbb{H}^{-}(\bar{t}_2,\bar{y}_2,\frac{1}{\epsilon}(\bar{y}_1 - \bar{y}_2) - \alpha \bar{y}_2 - \zeta_{y_2}) \geq 0.
\end{align*}
From (H.1)-(H.3), (\ref{eq_45}), (\ref{eq_46}), and Lemma \ref{Lemma_Appendix_D_1},
\begin{align*}
%\label{eq_D_5}
\sigma   	  
& \leq 2 \delta   + \Bigl | \mathbb{H}^{-}(\bar{y}_1,\frac{1}{\epsilon} (\bar{y}_1 - \bar{y}_2) + \alpha \bar{y}_1 + \zeta_{y_1})  - \mathbb{H}^{-}(\bar{y}_1,\frac{1}{\epsilon}(\bar{y}_1 - \bar{y}_2) - \alpha \bar{y}_2 - \zeta_{y_2}) \Bigr |  \\
&~~~ \qquad + \Bigl | \mathbb{H}^{-}(\bar{y}_1,\frac{1}{\epsilon}(\bar{y}_1 - \bar{y}_2) - \alpha \bar{y}_2 - \zeta_{y_2})  - \mathbb{H}^{-}(\bar{y}_2,\frac{1}{\epsilon}(\bar{y}_1 - \bar{y}_2) - \alpha \bar{y}_2 - \zeta_{y_2}) \Bigr | \nonumber \\
%
%
%& \leq 2\delta + c \Bigl \|\alpha (\bar{y}_1 + \bar{y}_2) + \zeta_{y_1} + \zeta_{y_2}  \Bigr \|_{\mathcal{L}_2}  \\
%&~~~ + c(1 + \Bigl \|\frac{1}{\epsilon}(\bar{y}_1 - \bar{y}_2) - \alpha \bar{y}_2 - \zeta_{y_2}  \Bigr \|_{\mathcal{L}_2} )\\
%&~~~~~~ \times \| \bar{y}_1 - \bar{y}_2 \|_{\mathcal{L}_2} \\
%& \leq 2 \delta +  c \alpha \Bigl (\frac{\delta}{\alpha} + \frac{1}{\alpha^{1/2}} \Bigr )  + c \delta    + c \epsilon^{1/2} + c \epsilon^{1/2} \delta^{1/2} \Bigl (\frac{\delta}{\alpha} + \frac{1}{\alpha^{1/2}} \Bigr )^{1/2}  \\
%&~~~ + c \epsilon + c \epsilon \delta \Bigl (\frac{\delta}{\alpha} + \frac{1}{\alpha^{1/2}} \Bigr )  + c \delta \epsilon^{1/2} + c \epsilon^{1/2} \delta^{3/2} \Bigl (\frac{\delta}{\alpha} + \frac{1}{\alpha^{1/2}} \Bigr )^{1/2}. \nonumber \\
& \leq 2 \delta +  c \alpha \Bigl (\frac{\delta}{\alpha} + \frac{1}{\alpha^{1/2}} \Bigr )^2  + c \delta^2     + c \epsilon + c \epsilon \delta \Bigl (\frac{\delta}{\alpha} + \frac{1}{\alpha^{1/2}} \Bigr )  + c \delta \epsilon^{1/2}.  \nonumber
\end{align*}
First, let $\delta \downarrow 0$ and then $\alpha \downarrow 0$ and $\epsilon \downarrow 0$. Then, we can easily get a contradiction, since $\sigma > 0$. This shows that we can select small positive $\delta$, $\alpha$ and $\epsilon$ such that either $\bar{t}_1 =0$ or $\bar{t}_2 = 0$. 

Let us assume that $\bar{t}_1 =0$. Then, the maximum property of $\Phi^\prime$ and its definition yield
\begin{align*}
\Phi^\prime(t,y,t,y) &\leq 	
\Phi(0,\bar{y}_1,\bar{t}_2,\bar{y}_2) - \zeta_{t_2} \bar{t}_2 - \mathbb{E}[\langle \zeta_{y_1}, \bar{y}_1 \rangle] - \mathbb{E}[\langle \zeta_{y_2}, \bar{y}_2 \rangle],
\end{align*}
which implies
\begin{align}
\label{eq_49}
\Phi(t,y,t,y) \leq &	 \Phi(0,\bar{y}_1,\bar{t}_2,\bar{y}_2)  - \zeta_{t_2} \bar{t}_2 - \mathbb{E}[\langle \zeta_{y_1}, \bar{y}_1 \rangle] - \mathbb{E}[\langle \zeta_{y_2}, \bar{y}_2 \rangle] \\
& + (\zeta_{t_1} + \zeta_{t_2}) t + \mathbb{E}[\langle \zeta_{y_1} + \zeta_{y_2}, y \rangle]. \nonumber
\end{align}
Since $\bar{t}_1 = 0$ and $\mathbb{L}_1(0,y) = \mathbb{L}_2(0,y)$, we have
\begin{align}
\label{eq_D_7}
\Phi(0,\bar{y}_1,\bar{t}_2,\bar{y}_2) 
%& \leq \mathbb{L}_1(0,\bar{y}_1) - \mathbb{L}_2(\bar{t}_2,\bar{y}_2)  \\
%& = \mathbb{L}_2(0,\bar{y}_1) - \mathbb{L}_2(\bar{t}_2,\bar{y}_2) \nonumber \\
%& \leq c (|\bar{t}_2| + \| \bar{y}_1 - \bar{y}_2\|_{\mathcal{L}_2}) \nonumber \\
& \leq c \epsilon^{1/2} + c \epsilon^{1/2} \delta^{1/2} \Bigl (\frac{\delta}{\alpha} + \frac{1}{\alpha^{1/2}} \Bigr )^{1/2}, 
\end{align}
where the inequality follows from the Lipschitz property and (\ref{eq_46}). On the other hand, in view of (\ref{eq_45}) and Cauchy-Schwarz inequality,
\begin{align}
\label{eq_D_8}
|\zeta_{t_2} \bar{t}_2 + \mathbb{E}[\langle \zeta_{y_1},\bar{y}_1 \rangle ]	 + \mathbb{E}[\langle \zeta_{y_2}, \bar{y}_2 \rangle]|  & \leq c \delta  + c \delta  \Bigl (\frac{\delta}{\alpha} + \frac{1}{\alpha^{1/2}} \Bigr ) \\
\label{eq_D_9}
|(\zeta_{t_1} + \zeta_{t_2}) t + \mathbb{E}[\langle \zeta_{y_1} + \zeta_{y_2}, y \rangle]| &\leq c \delta.	
\end{align}
By first letting $\epsilon \downarrow 0$, and then $\delta \downarrow 0$ and $\alpha \downarrow 0$ in (\ref{eq_D_7})-(\ref{eq_D_9}), from (\ref{eq_49}) and the definition of $\Phi$, we have
\begin{align*}
\mathbb{L}_1(t,y) - \mathbb{L}_2(t,y)	\leq 0,
\end{align*}
which leads to the desired result in (\ref{eq_44}). Then $L_1(t,\nu) \leq L_2(t,\nu)$ for $(t,\nu) \in [0,T] \times (\mathcal{P}_2(\mathbb{R}^{n}))^2$ in (\ref{eq_43}) follows from Lemma \ref{Lemma_1}. The proofs for $\mathbb{M}_1 \leq \mathbb{M}_2$ in (\ref{eq_44}) and $M_1 \leq M_2$ in (\ref{eq_43}) are similar. This completes the proof.
\end{proof}

\section{Proof of Corollaries \ref{Corollary_1} and \ref{Corollary_2}}\label{Appendix_E}

\begin{proof}[Proof of Corollary \ref{Corollary_1}]
	Suppose that $\mathbb{L}_1$ and $\mathbb{L}_2$ are value functions that are viscosity solutions to (\ref{eq_36}). In view of (\ref{eq_44}) in Theorem \ref{Theorem_3}, Lemma \ref{Lemma_2}, and the definition of the viscosity solution, we have $\mathbb{L}_1 \leq \mathbb{L}_2$ and $\mathbb{L}_2 \leq \mathbb{L}_1$, which implies that $\mathbb{L} := \mathbb{L}_1 = \mathbb{L}_2$. By Proposition \ref{Proposition_1}, $\mathbb{L}$ is the corresponding lifted lower value function. The proof of the remaining part is similar. This completes the proof.
\end{proof}

%The proof of Corollary \ref{Corollary_2} is as follows.
\begin{proof}[Proof of Corollary \ref{Corollary_2}]
	Set $H := H^{-} = H^{+}$ in (\ref{eq_32}) and (\ref{eq_33}), and $\mathbb{H} := \mathbb{H}^{-} = \mathbb{H}^{+}$ in (\ref{eq_36}) and (\ref{eq_37}). Then, (\ref{eq_32}) and (\ref{eq_33}) become identical HJI equations, and so do (\ref{eq_36}) and (\ref{eq_37}). From Lemmas \ref{Lemma_1} and \ref{Lemma_2}, Proposition \ref{Proposition_1} and Theorem \ref{Theorem_2}, together with the uniqueness result in Corollary \ref{Corollary_1}, we have $L(t,\nu) = \mathbb{L}(t,y) = \mathbb{M}(t,y) = M(t,\nu)  $ for $(t,\nu) \in [0,T] \times (\mathcal{P}_2(\mathbb{R}^{n}))^2$ and $(t,y) \in [0,T] \times (\mathcal{L}_2(\Omega, \mathbb{R}^{n}))^2$, which is the value of the ZSDG and is the unique solution to the HJI equation. This completes the proof.
\end{proof}

\end{appendix}

\bibliographystyle{IEEEtran}       
\bibliography{researches_1.bib} 

\begin{thebibliography}{10}
\providecommand{\url}[1]{#1}
\csname url@samestyle\endcsname
\providecommand{\newblock}{\relax}
\providecommand{\bibinfo}[2]{#2}
\providecommand{\BIBentrySTDinterwordspacing}{\spaceskip=0pt\relax}
\providecommand{\BIBentryALTinterwordstretchfactor}{4}
\providecommand{\BIBentryALTinterwordspacing}{\spaceskip=\fontdimen2\font plus
\BIBentryALTinterwordstretchfactor\fontdimen3\font minus
  \fontdimen4\font\relax}
\providecommand{\BIBforeignlanguage}[2]{{%
\expandafter\ifx\csname l@#1\endcsname\relax
\typeout{** WARNING: IEEEtran.bst: No hyphenation pattern has been}%
\typeout{** loaded for the language `#1'. Using the pattern for}%
\typeout{** the default language instead.}%
\else
\language=\csname l@#1\endcsname
\fi
#2}}
\providecommand{\BIBdecl}{\relax}
\BIBdecl

\bibitem{Basar2}
T.~Ba\c{s}ar and G.~J. Olsder, \emph{Dynamic Noncooperative Game Theory},
  2nd~ed.\hskip 1em plus 0.5em minus 0.4em\relax SIAM, 1999.

\bibitem{Bardi_book_1997}
M.~Bardi and I.~Capuzzo-Dolcetta, \emph{Optimal Control and Viscosity
  Solustions of $\text{Hamilton-Jabobi-Bellman}$ Equations}.\hskip 1em plus
  0.5em minus 0.4em\relax Birkh\"{a}user, 1997.

\bibitem{Bensoussan_book_2007}
A.~Bensoussan, G.~Da~Prato, M.~C. Delfour, and S.~Mitter, \emph{Representation
  and Control of Infinite Dimensional Systems}, 2nd~ed.\hskip 1em plus 0.5em
  minus 0.4em\relax Birkhauser, 2007.

\bibitem{Buckdahn_Cardaliaguet_DGAA_2011}
R.~Buckdahn, P.~Cardaliaguet, and M.~Quincampoix, ``Some recent aspects of
  differential game theory,'' \emph{Dynamic Games and Applications}, vol.~1,
  pp. 74--114, 2011.

\bibitem{Basar_Zaccour_book}
T.~Ba\c{s}ar and G.~Zaccour, Eds., \emph{Handbook of Dynamic Game Theory,
  Volumes I and II}.\hskip 1em plus 0.5em minus 0.4em\relax Springer, 2018.

\bibitem{Isaacs_book}
R.~Isaacs, \emph{Differential Games}.\hskip 1em plus 0.5em minus 0.4em\relax
  Dover Publications, 1965.

\bibitem{Elliott_1972}
R.~J. Elliott and N.~J. Kalton, \emph{Existence of value in differential
  games}.\hskip 1em plus 0.5em minus 0.4em\relax Memoirs of the Ametican
  Mathematical Society, 1972, vol. 126.

\bibitem{Evans_1984}
L.~C. Evans and P.~E. Souganidis, ``Differential games and representation
  formulas for solutions of $\text{Hamilton-Jabobi-Isaacs}$ equations,''
  \emph{Indiana Univ. Math. J}, vol.~33, pp. 293--314, 1984.

\bibitem{Fleming_1989}
W.~H. Fleming and P.~E. Souganidis, ``On the existence of value functions of
  two-player zero-sum stochastic differential games,'' \emph{Indiana Univ.
  Math. J}, vol.~38, pp. 293--314, 1989.

\bibitem{Friedman_JDE_1970}
A.~Friedman, ``Existence of value and of saddle points for differential games
  of pursuit and evasion,'' \emph{Journal of Differential Equations}, vol.~7,
  pp. 92--110, 1970.

\bibitem{Ghosh_JOTA_2004}
M.~K. Ghosh and A.~J. Shaiju, ``Existence of value and saddle point
  ininfinite-dimensional differential games,'' \emph{Journal of Optimization
  Theory and Applications}, vol. 121, no.~2, pp. 301--325, 2004.

\bibitem{Buckdahn_SICON_2008}
R.~Buckdahn and J.~Li, ``Stochastic differential games and viscosity solutions
  of $\text{Hamilton-Jabobi-Bellman-Isaacs}$ equations,'' \emph{SIAM Journal on
  Control and Optimization}, vol.~47, no.~4, pp. 444--475, 2008.

\bibitem{Li_AMO_2015}
J.~Li and Q.~Wei, ``Stochastic differential games for fully coupled
  $\text{FBSDEs}$ with jumps,'' \emph{Applied Mathematics and Optimization},
  vol.~71, pp. 411--448, 2015.

\bibitem{Qui_ESIAM_2013}
H.~Qui and J.~Yong, ``$\text{Hamilton-Jacobi}$ equations and two-person
  zero-sum differential games with unbounded controls,'' \emph{ESIAM: Control,
  Optimization and Calculus of Variations}, vol.~19, pp. 404--437, 2013.

\bibitem{Pham_SICON_2014}
T.~Pham and J.~Zhang, ``Two-person zero-sum game in weak formulation and path
  dependent $\text{Bellman-Isaacs Equation}$,'' \emph{SIAM Journal on Control
  and Optimization}, vol.~52, no.~4, pp. 2090--2121, 2014.

\bibitem{Li_MIn_SICON_2016}
J.~Li and H.~Min, ``Weak solutions of mean field stochastic differential
  equations and applications to zero-sum stochastic differential games.''
  \emph{SIAM Journal on Control and Optimization}, vol.~54, no.~3, pp.
  1826--1858, 2016.

\bibitem{Djehiche_AMO_2018}
B.~Djehiche and S.~Hamadene, ``Optimal control and zero-sum stochastic
  differential game problems of mean-field type,'' \emph{Applied Mathematics
  and Optimization}, pp. 1--28, 2016.

\bibitem{Averboukh_DGAA_2018}
Y.~Averboukh, ``$\text{Krasovskii-Subbotin}$ approach to mean field type
  differential games,'' \emph{Dynamic Games and Applications}, pp. 1--21, 2018,
  https://doi.org/10.1007/s13235-018-0282-6.

\bibitem{Zhang_SICON_2005}
P.~Zhang, ``Some results on two-person zero-sum linear quadratic differential
  games,'' \emph{SIAM Journal on Control and Optimization}, vol.~43, no.~6, pp.
  2157--2165, 2005.

\bibitem{Delfour_SICON_2007}
M.~C. Delfour, ``Linear quadratic differential games: Saddle point and
  $\text{Riccati}$ differential equations,'' \emph{SIAM Journal on Control and
  Optimization}, vol.~46, no.~2, pp. 750--774, 2007.

\bibitem{YU_SICON_2015}
Z.~Yu, ``An optimal feedback control-strategy pair for zero-sum
  linear-quadratic stochastic differential game: the $\text{R}$iccati equation
  approach,'' \emph{SIAM Journal on Control and Optimization}, vol.~53, no.~4,
  pp. 2141--2167, 2015.

\bibitem{Sun_SICON_2016}
J.~Sun, X.~Li, and J.~Yong, ``Open-loop and closed-loop solvability for
  stochastic linear quadratic optimal control problems,'' \emph{SIAM Journal on
  Control and Optimization}, vol.~54, no.~5, pp. 2274--2308, 2016.

\bibitem{Moon_TAC_2019_Markov}
J.~Moon, ``A sufficient condition for linear-quadratic stochastic zero-sum
  differential games for $\text{Markov}$ jump systems,'' \emph{IEEE
  Transactions on Automatic Control}, vol.~64, no.~4, pp. 1619--1626, 2019.

\bibitem{Moon_TAC_Risk_2019}
J.~Moon, T.~Duncan, and T.~Ba\c{s}ar, ``Risk-sensitive zero-sum differential
  games,'' \emph{IEEE Transactions on Automatic Control}, vol.~64, no.~4, pp.
  1503--1518, 2019.

\bibitem{5446378}
G.~Wang and Z.~Yu, ``A $\text{P}$ontryagin's maximum principle for non-zero sum
  differential games of $\text{BSDEs}$ with applications,'' \emph{IEEE
  Transactions on Automatic Control}, vol.~55, no.~7, pp. 1742--1747, 2010.

\bibitem{Mitchell_TAC_2005}
I.~M. Mitchell, A.~M. Bayen, and C.~J. Tomlin, ``A time-dependent
  $\text{Hamilton-Jacobi}$ formulation of reachable sets for continuous dynamic
  games,'' \emph{IEEE Transactions on Automatic Control}, vol.~50, no.~7, pp.
  947--957, 2005.

\bibitem{Margellos_TAC_2011}
K.~Margellos and J.~Lygeros, ``$\text{Hamilton-Jacobi}$ formulation for
  reach-avoid differential games,'' \emph{IEEE Transactions on Automatic
  Control}, vol.~56, no.~8, pp. 1849--1861, 2011.

\bibitem{Yong_Book_2015}
J.~Yong, \emph{Differential Games: A Concise Introduction}.\hskip 1em plus
  0.5em minus 0.4em\relax World Scientific, 2015.

\bibitem{Yong_book}
J.~Yong and X.~Y. Zhou, \emph{Stochastic Controls: Hamiltonian Systems and HJB
  Equations}.\hskip 1em plus 0.5em minus 0.4em\relax Springer, 1999.

\bibitem{Cardaliaguet}
P.~Cardaliaguet, ``Notes on mean field games,'' Jan 2012, $\text{T}$echnical
  Report.

\bibitem{Carmona_book_2018}
R.~Carmona and F.~Delarue, \emph{Probabilistic Theory of Mean Field Games with
  Applications I}.\hskip 1em plus 0.5em minus 0.4em\relax Springer, 2018.

\bibitem{Cardaliaguet_IGTR_2008}
P.~Cardaliaguet and M.~Quincampoix, ``Deterministic differential games under
  probability knowledge of initial condition,'' \emph{International Game Theory
  Review}, vol.~10, no.~1, pp. 1--16, 2008.

\bibitem{Cosso_JMPA_2018}
A.~Cosso and H.~Pham, ``Zero-sum stochastic differential games of generalized
  $\text{McKean-Vlasov}$ type,'' \emph{Journal de Math\'ematiques Pures et
  Appliqu\'ees}, vol. 129, pp. 180--212, 2018.

\bibitem{Luenberger_book}
D.~G. Luenberger, \emph{Optimization by Vector Space Methods}.\hskip 1em plus
  0.5em minus 0.4em\relax John Wiley \& Sons, 1969.

\bibitem{Conway_2000_book}
J.~B. Conway, \emph{A Course in Functional Analysis}.\hskip 1em plus 0.5em
  minus 0.4em\relax Springer, 2000.

\bibitem{Rachev}
S.~T. Rachev and L.~Ruschendorf, \emph{Mass Transportation Theory: Volume I:
  Theory}.\hskip 1em plus 0.5em minus 0.4em\relax Springer, 1998.

\bibitem{Villani_book}
C.~Villani, \emph{Optimal Transport}.\hskip 1em plus 0.5em minus 0.4em\relax
  Springer, 2009.

\bibitem{Rene_anals_2015}
R.~Carmona and F.~Delarue, ``Forward-backward stochastic differential equations
  and controlled $\text{McKean-Vlasov}$ dynamics,'' \emph{The Annals of
  Probability}, vol.~43, no.~5, pp. 2647--2700, 2015.

\bibitem{Jourdain_2008}
B.~Jourdain, S.~Meleard, and W.~Woyczynski, ``Nonlinear $\text{SDEs}$ driven by
  $\text{Levy}$ processes and related $\text{PDEs}$,'' \emph{ALEA, Latin
  American Journal of Probability}, vol.~4, pp. 1--29, 2008.

\bibitem{Lasry}
J.~M. Lasry and P.~L. Lions, ``Mean field games,'' \emph{Jap. J. Math.},
  vol.~2, no.~1, pp. 229--260, 2007.

\bibitem{Andersson_AMO_2010}
D.~Andersson and B.~Djehiche, ``A maximum principle for $\text{SDEs}$ of
  mean-field type,'' \emph{Applied Mathematics and Optimization}, vol.~63,
  no.~3, pp. 341--356, 2010.

\bibitem{7047683}
B.~Djehiche, H.~Tembine, and R.~Tempone, ``A stochastic maximum principle for
  risk-sensitive mean-field type control,'' \emph{IEEE Transactions on
  Automatic Control}, vol.~60, no.~10, pp. 2640--2649, 2015.

\bibitem{Rene_SICON_2013}
R.~Carmona and F.~Delarue, ``Probabilistic analysis of mean-field games,''
  \emph{SIAM Journal on Control and Optimization}, vol.~51, no.~4, pp.
  2705--2734, 2013.

\bibitem{Yong_SICON_2013_MF}
J.~Yong, ``Linear-quadratic optimal control problems for mean-field stochastic
  differential equations,'' \emph{SIAM Journal on Control and Optimization},
  vol.~51, no.~4, pp. 2809--2838, 2013.

\bibitem{Bensoussan_Arxiv_2014}
A.~Bensoussan, K.~C.~J. Sung, S.~C.~P. Yam, and C.~P. Yung, ``Linear-quadratic
  mean field games,'' 2014, https://arxiv.org/abs/1404.5741.

\bibitem{Cardaliaguet_MFE_2018}
P.~Cardaliaguet and C.~A. Lehalle, ``Mean field game of controls and an
  application to trade crowding,'' \emph{Mathematics and Financial Economics},
  vol.~12, no.~3, pp. 335--363, 2018.

\bibitem{Moon_DGAA_2019}
J.~Moon and T.~Ba\c{s}ar, ``Risk-sensitive mean field games via the stochastic
  maximum principle,'' \emph{Dynamic Games and Applications}, vol.~9, pp.
  1100--1125, 2019.

\bibitem{Goodfellow_book}
I.~Goodfellow, Y.~Bengio, and A.~Courville, \emph{Deep Learning}.\hskip 1em
  plus 0.5em minus 0.4em\relax MIT Press, 2016.

\bibitem{Crandall_Lions_1983}
M.~G. Crandall and P.-L. Lions, ``Viscosity solutions of
  $\text{Hamilton-Jacobi}$ equations,'' \emph{Transactions on American
  Mathematical Society}, vol. 277, no.~1, pp. 1--42, 1983.

\bibitem{Crandall_Lions_MC_1984}
------, ``Two approximations of solutions of $\text{Hamilton-Jacobi}$
  equations,'' \emph{Mathematics of Computation}, vol.~43, no. 167, pp. 1--19,
  1984.

\bibitem{Basar_Rational_1989}
T.~Ba\c{s}ar, ``Some thoughts on rational expectations model, and alternative
  formulations,'' \emph{Computers and Mathematics with Applications}, vol.~18,
  no. 6/7, pp. 591--604, 1989.

\bibitem{Evans_Book}
L.~C. Evans, \emph{Partial Differential Equations}, 2nd~ed.\hskip 1em plus
  0.5em minus 0.4em\relax American Mathematical Society, 2010.

\bibitem{Stegall_1978}
C.~Stegall, ``Optimization of functions on certain subsets of $\text{Banach}$
  spaces,'' \emph{Mathematische Annalen}, vol. 236, no.~2, pp. 171--176, 1978.

\bibitem{Fabian_NA_2007}
M.~Fabian and C.~Finet, ``On $\text{Stegall's}$ smooth variational principle,''
  \emph{Nonlinear Analysis}, vol.~66, p.~5, 2007.

\end{thebibliography}

\address{Jun Moon \\ School of Electrical and Computer Engineering \\  University of Seoul \\
	Seoul, South Korea, 02504\\
\email{jmoon12@uos.ac.kr}}

\address{Tamer Ba\c{s}ar \\ Coordinated Science Laboratory \\
 University of Illinois at Urbana-Champaign \\
 Urbana, Illinois, USA, 61801 \\
\email{basar1@illinois.edu} 
 }
\end{document}